\documentclass{article}

\usepackage{arxiv}

\usepackage{lipsum}
\usepackage{amsfonts}
\usepackage{graphicx}
\usepackage{epstopdf}
\usepackage{algorithmic}
\usepackage{booktabs}
\usepackage{amsmath}
\usepackage{hyperref}
\usepackage{amsthm}
\usepackage{algorithm}

\usepackage{graphics,graphicx}
\usepackage{xcolor}
\usepackage{graphicx, caption, subcaption}

\newtheorem{theorem}{Theorem}[section]
\newtheorem{corollary}{Corollary}[theorem]
\newtheorem{lemma}[theorem]{Lemma}
\newtheorem{proposition}[theorem]{Proposition}
\newtheorem{definition}{Definition}[section]
\newtheorem{example}{Example}[section]

\newcommand{\R}{\mathbb{R}}

\newcommand{\D}{\mathbb{D}}

\newcommand{\Y}{\mathcal{Y}}
\newcommand{\X}{\mathcal{X}}

\newcommand{\y}{\boldsymbol{y}}
\newcommand{\x}{\boldsymbol{x}}
\newcommand{\e}{\boldsymbol{e}}
\newcommand{\z}{\boldsymbol{z}}
\newcommand{\bv}{\boldsymbol{v}}
\newcommand{\bu}{\boldsymbol{u}}
\newcommand{\A}{\boldsymbol{A}}
\newcommand{\I}{\boldsymbol{I}}

\DeclareMathOperator*\bigcircop{\bigcirc}

\title{To be or not to be stable, that is the question: understanding neural networks for inverse problems}

\author{Davide Evangelista \\
Department of Mathematics, \\ University of Bologna, Italy \\ \texttt{davide.evangelista5@unibo.it}. \\
\And Elena Loli Piccolomini \\
Department of Computer Science and Engineering, \\ University of Bologna, Italy.
\And Elena Morotti \\
Department of Political and Social Sciences, \\ University of Bologna, Italy. \\
\And James Nagy \\
Department of Mathematics, \\ Emory University, Atlanta.}

\date{}

\begin{document}
\maketitle

\begin{abstract}
    The solution of linear inverse problems arising, for example, in signal and image processing is a challenging problem since the ill-conditioning amplifies, in the solution, the noise present in the data. Recently introduced algorithms based on deep learning overwhelm the more traditional model-based approaches in performance, but they typically suffer from instability with respect to data perturbation. In this paper, we theoretically analyze the trade-off between stability and accuracy of neural networks, when used to solve linear imaging inverse problems for not under-determined cases. Moreover, we propose different supervised and unsupervised solutions to increase the network stability and maintain a good accuracy, by means of regularization properties inherited from a model-based iterative scheme during the network training and pre-processing stabilizing operator in the neural networks.
    Extensive numerical experiments on image deblurring confirm the theoretical results and the effectiveness of the proposed deep learning-based approaches to handle noise on the data.
\end{abstract}


\section{Introduction\label{sec:intro}}

Linear inverse problems of the form:
\begin{align}\label{eq:disc_fred_eq}
    \y = \A\x,
\end{align}
where $\A \in \R^{m \times n}$ is a full-rank matrix discretizing a linear operator,  $\x \in \R^{n}$ and  $\y \in \R^{m}$ with $m \geq n$, arise in various image processing tasks, such as deblurring or tomographic reconstruction \cite{hansen1998rank, hansen2010discrete, mueller2012linear}. 
It is well-known that in these applications, equation \eqref{eq:disc_fred_eq} represents the discretization of an ill-posed problem.
Following the well-known Hadamard definition, a problem is ill-posed if either a solution does not exist, the solution is not unique or it does not continuously depend on the data $\y$. 
In the case considered in \eqref{eq:disc_fred_eq} the third condition holds, thereby the computation of the solution becomes very challenging when noise affects the data.
In this work, we consider data corrupted by Gaussian noise, i.e.:
\begin{equation}\label{eq:forward_problem}
    \y^\delta = \A\x^{gt} + \e, \qquad \e \sim \mathcal{N}(\mathbf{0}, \delta^2 \I);
\end{equation}
where $\delta$ denotes the standard deviation of the white additive Gaussian noise, $\I$ is the identity matrix, and $\x^{gt}$ is the ground truth, clean image.

Traditional regularization approaches tackle problem \eqref{eq:forward_problem} as the minimization of an objective function containing a data-fit term and a regularization prior, with possible further constraints on the solution \cite{bertero2021introduction, engl1996regularization}. 
These terms theoretically grant stability, but, in general, the computational time required by solvers is high and it may be necessary to choose many parameters, tuning them by trial and error on the data.

In the last few years, neural networks have been introduced with great success for the solution of problem \eqref{eq:forward_problem},  since they are capable of achieving greater accuracy than iterative regularized methods \cite{on_deep_learning_for_inverse_problems, cascarano2022plug, solvability_of_inverse_problems_in_medical_imaging}. 
However, noise-related issues still persist, as their high accuracy is obtained at the expense of robustness against noise in the input data. 
Specifically, these networks frequently yield suboptimal results when applied to data contaminated with noise that differs from that encountered during the training phase. This tendency is commonly referred to as network {\em instability}.\\
Some authors have already studied the behavior of neural networks in the presence of noise on the data, focusing on the solution of under-determined imaging inverse problems (i.e. when $m<n$ in equation \eqref{eq:disc_fred_eq}) \cite{robustness_included, stabilizing_deep_tomographic_reconstruction_A, stabilizing_deep_tomographic_reconstruction_B, antun2020instabilities, darestani2021measuring, huang2018some, johnson2021evaluation, morshuis2022adversarial, muckley2021results, pal2022review, shimron2022implicit, yu2022validation, zhang2021instabilities, troublesome_kernel, can_stable_network_be_computed}. 
We note that the paper \cite{troublesome_kernel} offers a comprehensive bibliography on this topic, with the authors remarking that``{\it stability implies a universal barrier on performance}''.
However, to the best of our knowledge, a mathematically grounded understanding is still lacking and no works address the case of $m\ge n$.\\

\paragraph{Contributions} In this work, we look at neural networks as solvers of discrete ill-posed problems, and we contribute to the state-of-art as follows.

Firstly, we adapt the regularization theory presented by Engl at al. in \cite{regularization_of_inverse_problems} for solving discrete inverse problems through neural networks. It is noteworthy that Engl et al.~examined regularization in Hilbert spaces, while our focus is on discrete inverse problems. Prior to introducing neural networks as solvers, we present a more general theory encompassing a broader class of functions, termed {\em reconstructors},  designed for addressing problem \eqref{eq:forward_problem}. Within this framework, we first formalize the two fundamental concepts of $\eta^{-1}$-accuracy and $\epsilon$-stability, and then we present significant findings for a class of functions called {\em stabilizers}. We establish a mathematical relationship that quantifies the trade-off between stability and accuracy, demonstrating that enhancing a solver's stability is impossible without compromising its accuracy.
In this theoretical approach, neural networks have been analyzed as formal mathematical operators, shedding light on their wildly discussed 'black-box' nature/misinterpretation.

Secondly, we propose a new ground truth-free scheme for reconstructors based on neural networks. We refer to this approach as the REgularized Neural Network (ReNN), as the target images used in the training procedure are solutions of \eqref{eq:forward_problem} computed through a regularization method.
Beyond being more stable than commonly used neural networks as reconstructors, it is applicable in scenarios where collecting a set of ground-truth solutions is challenging or impossible, such as in medical imaging.

Finally,  
we present a novel stabilization strategy tailored for deep learning-based solvers, which incorporates a stabilizer within a pre-processing operator plugged into a neural network reconstructor. This approach demonstrates substantial efficacy in handling elevated noise levels in data. We have termed this methodology STabilized Neural Network (StNN). Furthermore, when integrated with the ReNN scheme, it evolves into the StReNN framework.\\
 

\paragraph{Structure of the paper} The paper is structured as follows. In Section \ref{sec:Reconstr}, we introduce theoretical concepts related to reconstructors for solving an inverse problem of the form presented in \eqref{eq:forward_problem}. In Section \ref{sec:stabilizers}, we present stabilizers and elucidate their effectiveness by stating their properties, then Section \ref{sec:nn_as_reconstructors} is dedicated to reconstructors based on neural networks and presents our proposals. Following that, in Section \ref{sec:expsetup}, we describe our experimental setup, whereas Section \ref{sec:results} showcases numerical results pertaining to deblurring and denoising. Finally, Section \ref{sec:concl} comprises conclusions and outlines potential directions for future work.
\section{Reconstructors for the solution of linear inverse problems}\label{sec:Reconstr}

This section establishes the theoretical background of the manuscript, providing essential definitions and preliminary results.
To improve the readability of the work, however, we start by introducing the notation we will use throughout the paper. 
We always consider  $\x^{gt}$ to lie in a subset $\X$ of $\R^n$, the set of admissible data. We denote as $\mathcal{Y} = Rg(\A, \X)$ the range of $\A$ over $\X$, where $\A$ is a continuous linear operator. We assume $\Y$ to be dense-in-itself (i.e. with no isolated point) so that, for any admissible $\x^{gt} \in \X$ and any neighborhood $V$ of $\y = \A \x^{gt}$, there is at least an $\x' \in \X$, $\x' \neq \x^{gt}$, such that $ \y' = \A \x' \in V$. When $\x \in \X \subseteq \R^n$ or $\y \in \Y \subseteq \R^m$, then $||\x||$ and $|| \y ||$ will be Euclidian norms. For any $\epsilon > 0$, we also define $\Y^\epsilon = \{ \y + \e; \> \y \in \Y, ||\e||\leq \epsilon \}$. 
With the following definitions, we can formalize the concept of reconstructor to solve problem \eqref{eq:disc_fred_eq} accurately.

\begin{definition}
    Any continuous function $\Psi: \R^{m} \to \R^{n}$, mapping $\y$ 
    to $\x=\Psi(\y)$, is called a reconstructor. 
\end{definition}


\begin{definition}
    A reconstructor $\Psi: \R^{m} \to \R^{n}$ is said to be $\eta^{-1}$-accurate, with $\eta>0$, if:
    \begin{align*}
        \eta = \sup_{\x^{gt} \in \X} || \Psi(\A\x^{gt}) - \x^{gt} ||.
    \end{align*}
    We  define the set 
     $ \  \mathcal{R}_\eta = \{ \Psi: \R^m \to \R^n; \Psi \text{ is a reconstructor with accuracy } \eta^{-1} \}. $
\end{definition}

We observe that without any other restriction, $\eta$ could be infinite. To avoid any issue, we will always consider reconstructors with finite $\eta$ in the following.

\begin{example}\label{ex:pseudoinverse_reconstructor}
An $\infty-$accurate reconstructor of problem \eqref{eq:disc_fred_eq} is given by: 
\begin{equation*}
  \Psi^\dagger(\y) = \A^\dagger \y = (\A^* \A)^{-1} \A^* \y,
\end{equation*}
where  $\A^\dagger$ is the pseudo-inverse matrix, $\A^*$ is the transpose of $\A$, and the last equality holds since $\A$ is assumed to be full-rank. In this case $\Psi^\dagger: \R^m \to \R^n$ is $\infty$-accurate, as: 
$$
|| \Psi^\dagger(\A\x^{gt}) - \x^{gt} || = || (\A^* \A)^{-1} \A^* (\A \x^{gt}) - \x^{gt}|| = || (\A^* \A)^{-1} (\A^* \A) \x^{gt} - \x^{gt} || = 0.
$$
\end{example}

However, reconstructors are rarely applied to noise-free data, hence a focus on the robustness of reconstructors with respect to noise is necessary.

\begin{definition}\label{eq:definition_of_stability_constant}
    Let $\epsilon>0$ and $\Psi$ be an $\eta^{-1}$-accurate reconstructor applied to problem \eqref{eq:forward_problem}. We define the $\epsilon$-stability constant $C_\Psi^\epsilon$  of $\Psi$ as:
    \begin{equation*}
    C^\epsilon_\Psi = \sup_{\substack{\x^{gt} \in \X \\ || \e || \leq \epsilon}} \frac{|| \Psi(\A\x^{gt} + \e) - \x^{gt} || - \eta}{|| \e ||}.
    \end{equation*}
\end{definition}

We will consider in the following the realistic case of $C_\Psi^\epsilon < \infty$.

\begin{definition}
    The reconstructor  $\Psi$ is said to be $\epsilon$-stable for a given $\epsilon>0$ if $C_\Psi^\epsilon \in [0, 1)$. Otherwise, $\Psi$ is said to be $\epsilon$-unstable.
\end{definition}

An $\epsilon$-stable reconstructor $\Psi$ does not amplify corruptions having norm less than $\epsilon$ (as graphically represented in Figure \ref{fig:stable_unstable}), since \eqref{eq:definition_of_stability_constant} implies:
$$
    || \Psi(\A\x^{gt} + \e) - \x^{gt} || \leq \eta + C^\epsilon_\Psi ||\e|| \qquad \forall \> \x^{gt} \in \X, \forall \> \e \in \R^m, ||\e||\leq \epsilon.
$$ 

\begin{definition}
   We define the {\it stability radius} $\rho$ of $\Psi$ as:
    \begin{equation*}
        \rho = \sup \{ \epsilon>0;\> C^\epsilon_\Psi \in [0, 1) \}\,.
    \end{equation*}
\end{definition}

\begin{figure}
    \centering
    \includegraphics[width=0.7\linewidth]{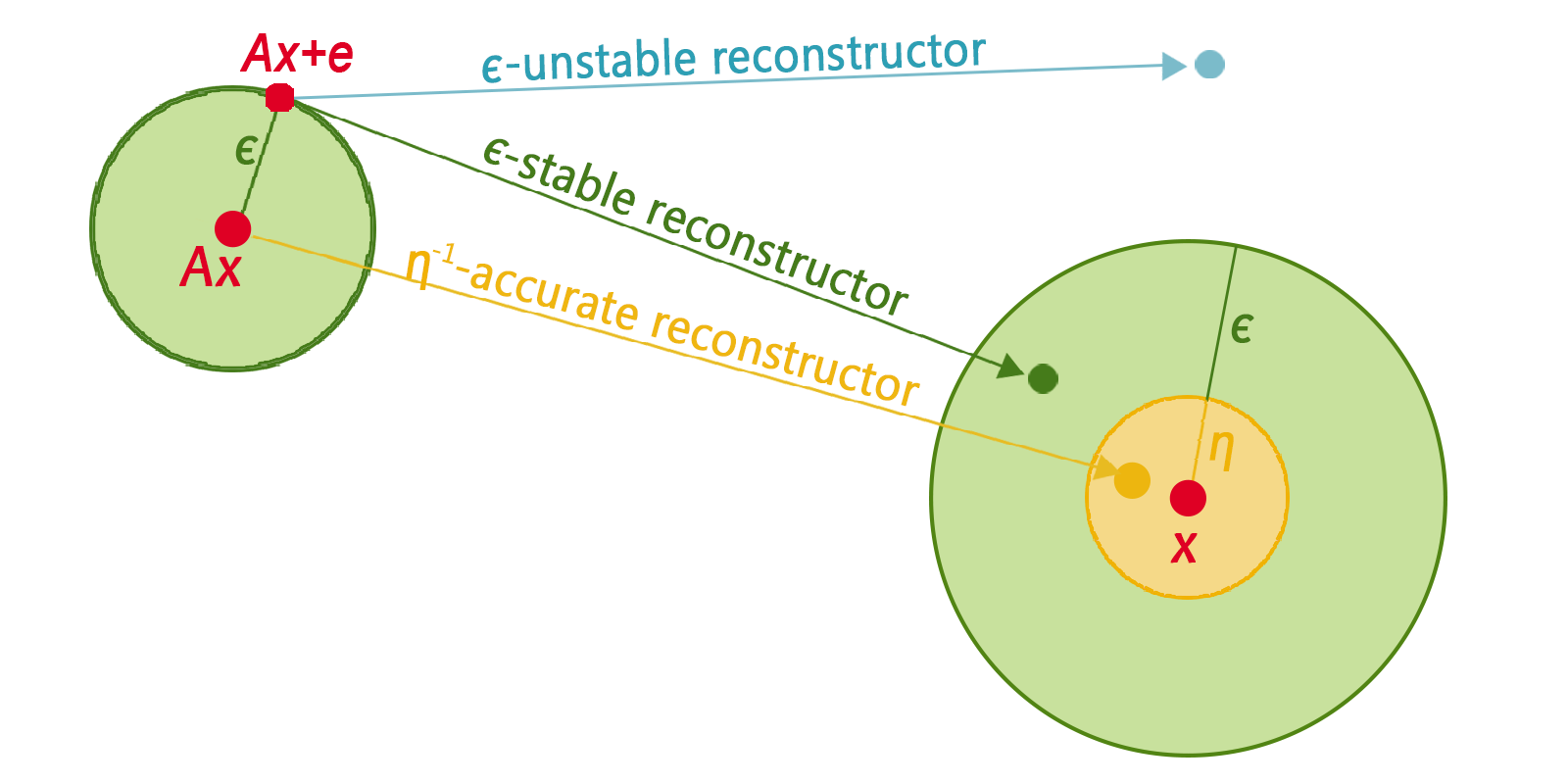}
    \caption{Graphical representation of the $\epsilon$-stability and $\epsilon$-instability  for an $\eta^{-1}$-accurate reconstructor. }
    \label{fig:stable_unstable}
\end{figure}

\begin{example}\label{example:constant_reconstructor}
A reconstructor with an infinite stability radius is the following. Given $\epsilon>0$, if $\mu$ is a probability distribution over $\X$ (for example, $\mu$ is the normalized Lebesgue distribution over $\X$), the reconstructor defined as:
$$
\Psi^{\X, \epsilon}(\y^\delta) = \int_\X \x \mu(d\x), \qquad \forall \y^\delta \in \Y^\epsilon
$$
is $\epsilon$-stable independently from the value of $\epsilon >0$. Indeed:
$$
|| \Psi^{\X, \epsilon}(\A\x^{gt} + \e) - \x^{gt} || = \Bigl\|  \int_\X \x \mu(d\x) - \x^{gt} \Bigr\| \leq \rho(\X),
$$
where $\rho(\X)$ is the radius of $\X$, defined as $\rho(\X) = \inf \{r > 0: \X \subseteq \mathcal{B}(\int_\X \x \mu(d\x); r)\}$. As a consequence the stability constant is infinite regardless $\epsilon$, and $\Psi^{\X(\y), \epsilon}$ has accuracy $ \rho(\X)^{-1}$.
\end{example}

\begin{example}
The pseudo-inverse reconstructor $\Psi^\dagger(\y)$ in  \eqref{ex:pseudoinverse_reconstructor} is unstable for any $\epsilon>0$ when $\A$ is ill-conditioned. Indeed: 
\begin{align*}
|| \Psi^\dagger(\A\x^{gt} + \e) - \x^{gt} || = || (\A^* \A)^{-1} (\A^* \A) \x^{gt} + (\A^* \A)^{-1} \A^* \e - \x^{gt} || \\ 
= || (\A^* \A)^{-1} \A^* \e ||.
\end{align*}
If $\A = \boldsymbol{U} \boldsymbol{\Sigma} \boldsymbol{V}^*$ is the Singular Value Decomposition (SVD) of $\A$, then:
$$
    (\A^* \A)^{-1} \A^* \e = (\boldsymbol{V} \boldsymbol{\Sigma}^2 \boldsymbol{V}^*)^{-1} \boldsymbol{V} \boldsymbol{\Sigma} \boldsymbol{U}^* \e = \boldsymbol{V} \boldsymbol{\Sigma}^\dagger \boldsymbol{U}^* \e = \sum_{i=1}^{n} \frac{\boldsymbol{u}_i^T \e}{\sigma_i} \boldsymbol{v}_i,
$$
which implies that $||(\A^* \A)^{-1} \A^* \e|| \gg ||\e||$ when $\A$ has singular values close to zero.
\end{example}

These examples shed light on a possible conflict between accuracy and stability for a given reconstructor $\Psi$. 
In the next paragraphs, we study this relationship.

\subsection{Accuracy vs. stability trade-off}
We can derive a relation between accuracy and stability, which becomes particularly interesting when $\A$ is ill-conditioned.

\begin{lemma}\label{lemma:stab_vs_acc_tradeoff}
    Let $\Psi: \R^m \to \R^n$ be an $\eta^{-1}$-accurate reconstructor. Then, for any $\x^{gt} \in \X$ and for any $\epsilon > 0$, $\exists \: \tilde{\e} \in \R^m$ with $|| \tilde{\e} || \leq \epsilon$ such that:
    \begin{equation}\label{eq:bad_corruption}
        || \Psi(\A\x^{gt} + \tilde{\e}) - \x^{gt} || \geq ||\A^\dagger \tilde{\e}|| - \eta\,.
    \end{equation}
\end{lemma}

\begin{proof}
    Since $\A\x^{gt} \in \Y$ for any $\x^{gt} \in \X$, and since $\Y$ has no isolated points, then for any $\epsilon > 0$ there is an $\tilde{\e} \in \R^m$ with $|| \tilde{\e} || \leq \epsilon$ such that $\A\x^{gt} + \tilde{\e} \in \Y$. Thus, $\exists \: \x' \in \X$ such that $\A\x^{gt} + \tilde{\e} = \A\x'$. Consequently:
    \begin{eqnarray*}
        || \Psi(\A\x^{gt} + \tilde{\e}) - \x^{gt} || = || \Psi(\A\x') - \x^{gt} || & \geq &  || \x' - \x^{gt} || - || \Psi(\A\x') - \x' || \\
        & \geq & || \x' - \x^{gt}|| - \eta\,.
    \end{eqnarray*}
    Since  $\A\x^{gt} + \tilde{\e} = \A\x'$ by construction, it holds that $\tilde{\e} = \A(\x^{gt} - \x')$, which implies that $\x^{gt} - \x' = \A^\dagger \tilde{\e}$. To conclude:
    \begin{align*}
        || \Psi(\A\x^{gt} + \tilde{\e}) - \x^{gt} || \geq || \x' - \x^{gt}|| - \eta = ||\A^\dagger \tilde{\e}|| - \eta\,.
    \end{align*}
\end{proof}

Since the corruption $\tilde{\e}$ such that the relationship \eqref{eq:bad_corruption} holds for some $\epsilon>0$ depends on $\x^{gt}$, for any $\x^{gt} \in \X$, we will consider the set:
\begin{equation}
    E(\x^{gt}) = \{ \e \in \R^m; \: \text{Equation \eqref{eq:bad_corruption} holds for }\e,\text{ for some } \epsilon>0 \}.
\end{equation}

\begin{theorem}[Trade-off Theorem]\label{thm:stab_vs_acc_tradeoff}
    Under the assumptions of Lemma \ref{lemma:stab_vs_acc_tradeoff} it holds that, for any $\x^{gt} \in \X$ and for any $\tilde{\e} \in E(\x^{gt})$ with $|| \tilde{\e} || \leq \epsilon$,
    \begin{equation}\label{eq:stab_vs_acc_tradeoff}
        C^\epsilon_\Psi \geq \frac{||\A^\dagger \tilde{\e}|| - 2\eta}{|| \tilde{\e} ||}.
    \end{equation}
\end{theorem}

\begin{proof}
    For any $\x^{gt} \in \X$,
    \begin{align*}
        C^\epsilon_\Psi = \sup_{\substack{\x \in \X \\|| \e || \leq \epsilon}} \frac{|| \Psi(\A\x + \e) - \x || - \eta}{|| \e ||} \geq \sup_{ || \e || \leq \epsilon} \frac{|| \Psi(\A\x^{gt} + \e) - \x^{gt} || - \eta}{|| \e ||}.
    \end{align*}
    If $\tilde{\e} \in E(\x^{gt})$, $||\tilde{\e}|| \leq \epsilon$, is a perturbation defined on Lemma \ref{lemma:stab_vs_acc_tradeoff}, we have:
    \begin{align*}
    \begin{split}
        C^\epsilon_\Psi &\geq \sup_{|| \e || \leq \epsilon} \frac{|| \Psi(\A\x^{gt} + \e) - \x^{gt} || - \eta}{|| \e ||} \\ &\geq \frac{|| \Psi(\A\x^{gt} + \tilde{\e}) - \x^{gt} || - \eta}{|| \tilde{\e} ||} \\ &\geq \frac{||\A^\dagger \tilde{\e}|| - 2\eta}{|| \tilde{\e} ||},
    \end{split}
    \end{align*}
    which concludes the proof.
\end{proof}

\begin{corollary}\label{cor:bound_on_stab_radius}
Given the assumptions of Theorem \ref{thm:stab_vs_acc_tradeoff}, if $\X = \R^n$, there is a constant $C(\A) > 0$ which depends only on $\A$, such that:
\begin{equation*} 
\rho \leq \frac{2}{\eta^{-1}C(\A)}.
\end{equation*}
\end{corollary}
\begin{proof}
    Consider a reconstructor $\Psi$. By Theorem \ref{thm:stab_vs_acc_tradeoff}, for any $\epsilon > 0$, any $\x^{gt} \in \X$, and any $\tilde{\e} \in E(\x^{gt})$ with $|| \tilde{\e} || \leq \epsilon$, 
    \begin{align}
        C^\epsilon_\Psi \geq \frac{||\A^\dagger \tilde{\e}|| - 2\eta}{|| \tilde{\e} ||}.
    \end{align}
    
    We first observe that, if $\X = \R^n$, then $E(\x^{gt}) = \Y := Rg(\A, \X)$ for any $\x^{gt} \in \X$. Indeed, $\tilde{\e} \in E(\x^{gt})$ if and only if there exists $\x' \in \X$ such that $\tilde{\e} = \A(\x^{gt} - \x')$. Since $\X = \R^n$ is closed under addition, then $\x^{gt} - \x' \in \X$, which implies that $\tilde{\e} \in \Y$, thus $E(\x^{gt}) \subseteq \Y$. Conversely, if $\y \in \Y$, then by definition there exists $\x \in \X$ such that $\y = \A \x$. By defining $\x' = \x^{gt} - \x$, then $\y = \A (\x^{gt} - \x')$, which implies that $\y \in E(\x^{gt})$ and consequently $E(\x^{gt}) = \Y$. 

    Now, let $\A = \boldsymbol{U} \boldsymbol{\Sigma} \boldsymbol{V}^*$ be the SVD of $\A$ and define $\tilde{\e} = \A \left( \frac{\epsilon}{\sigma_n} \bv_n \right)$, where $\sigma_n$ and $\bv_n$ are the smallest singular value of $\A$ and its associated right-singular vector, respectively. Note that $\tilde{\e} \in \Y = E(\x^{gt})$ by definition. Moreover:
    \begin{align*}
        \tilde{\e} & = \A \left( \frac{\epsilon}{\sigma_n} \bv_n \right) = \boldsymbol{U} \boldsymbol{\Sigma} \boldsymbol{V}^* \left( \frac{\epsilon}{\sigma_n} \bv_n \right)  = \frac{\epsilon}{\sigma_n} \sum_{i=1}^n \sigma_i \bu_i \left( \bv_i^T \bv_n \right) \\ &= \frac{\epsilon}{\sigma_n} \sigma_n \bu_n = \epsilon \bu_n,
    \end{align*}
    from which $|| \tilde{\e} || = \epsilon || \bu_n || = \epsilon \leq \epsilon$. 
    Consequently, \eqref{eq:stab_vs_acc_tradeoff} holds for $\tilde{\e}$. Additionally:
    \begin{align*}
        \A^\dagger \tilde{\e} = \A^\dagger \A \left( \frac{\epsilon}{\sigma_n} \bv_n \right) = \frac{\epsilon}{\sigma_n} \bv_n,
    \end{align*}
    hence $|| \A^\dagger \tilde{\e} || = \frac{\epsilon}{\sigma_n}$. Given that, \eqref{eq:stab_vs_acc_tradeoff} reads:
    \begin{align*}
        C^\epsilon_\Psi \geq \frac{||\A^\dagger \tilde{\e}|| - 2\eta}{|| \tilde{\e} ||} =  \frac{\frac{\epsilon}{\sigma_n} - 2\eta}{\epsilon}.
    \end{align*}
    As a consequence of the above relationship, if 
    $
        \frac{\frac{\epsilon}{\sigma_n} - 2\eta}{\epsilon} > 1,
    $ 
    then $C_\Psi^\epsilon > 1$, i.e. $\rho \leq \epsilon$. A simple computation shows that this holds if:
    \begin{align*}
        \epsilon < \frac{2}{\eta^{-1} \left(\frac{1-\sigma_n}{\sigma_n}\right)} = \frac{2}{\eta^{-1}C(\A)},
    \end{align*}
    concluding the proof by calling $C(\A) = \frac{1-\sigma_n}{\sigma_n}$.
\end{proof}

The relation in Corollary \ref{cor:bound_on_stab_radius} between the stability radius $\rho$ and the accuracy $\eta^{-1}$ suggests that there exists a trade-off between accuracy and stability, showing that a very accurate reconstructor is unstable for noise corruption larger than $ \frac{2}{\eta^{-1}C(\A)}$. We remark that for ill-conditioned problems $C(\A) = \frac{1 - \sigma_n}{\sigma_n}$ can be very large, making the radius potentially very small.

Similarly, Theorem \ref{thm:stab_vs_acc_tradeoff} shows that a reconstructor $\Psi$ can be $\epsilon$-stable only if its accuracy is bounded.

\begin{corollary}\label{cor:bound_on_acc}
    Given the assumptions of Theorem \ref{thm:stab_vs_acc_tradeoff}, there exists $\bar{\eta}(\A, \epsilon, \X) \in \R \cup \{ + \infty\}$, such that any reconstructor $\Psi$ with accuracy $\eta^{-1} \geq \bar{\eta}(\A, \epsilon, \X)^{-1}$ is $\epsilon$-unstable, i.e. $C_\Psi^\epsilon \geq 1$.\\
    Moreover, if $\X = \R^n$ and $\eta^{-1} \geq \frac{2}{C(\A)\epsilon}$, where $C(\A)=\frac{1 - \sigma_n}{\sigma_n}$, then $\Psi$ is $\epsilon$-unstable.
\end{corollary}

\begin{proof}
    From Theorem \ref{thm:stab_vs_acc_tradeoff}, 
    $\Psi$ is $\epsilon$-unstable for a given $\epsilon>0$ if
     $   \frac{|| \A^\dagger \tilde{\e} || - 2\eta}{|| \tilde{\e} ||} \geq 1. $
    Such condition holds if and only if:
    \begin{equation*}
        \eta \leq \frac{|| \A^\dagger \tilde{\e} || - || \tilde{\e} ||}{2}.
    \end{equation*}
    Thus, if $\exists \: \tilde{\e} \in E(\x^{gt})$ with $|| \tilde{\e} || \leq \epsilon$ such that $\eta \leq \frac{|| \A^\dagger \tilde{\e} || - || \tilde{\e} ||}{2}$, then $\Psi$ is $\epsilon$-unstable. In particular, if we define:
    \begin{equation}
        \bar{\eta}(\A, \epsilon, \X) = \sup_{\substack{\x^{gt} \in \X \\ \tilde{\e} \in E(\x^{gt}) \\ || \tilde{\e} || \leq \epsilon}} \frac{|| \A^\dagger \tilde{\e} || - || \tilde{\e} ||}{2},
    \end{equation}
    we get the result. Note that, in general, $\bar{\eta}(\A, \epsilon, \X)$ could be infinite. \\
    In the assumption of $\X = \R^n$, we proved in Corollary \ref{cor:bound_on_stab_radius} that for any $\epsilon > 0$ and any $\x^{gt} \in \X$, we can always choose $\tilde{\e} \in E(\x^{gt})$ with $|| \tilde{\e} || \leq \epsilon$ such that $|| \A^\dagger \tilde{\e} || - || \tilde{\e} || = \frac{1 - \sigma_n}{\sigma_n}\epsilon = C(\A)\epsilon$. Thus, $\Psi$ is $\epsilon$-unstable if:
    \begin{equation*}
        \eta \leq \frac{C(\A)\epsilon}{2},
    \end{equation*}
    which proves the corollary.
\end{proof}

\subsection{A sufficient condition for stability}
Whenever a reconstructor is (locally) Lipschitz continuous, we can also derive conditions assessing its stability. First of all, we recall the definition of locally Lipschitz continuous reconstructors.

\begin{definition} \label{eq:definition_local_lip_constant}
    Given $\Y \subseteq \R^m$ and $\epsilon>0$, we define the $\epsilon$-Lipschitz (also called local Lipschitz) constant of $\Psi$ over $\Y$ as:
    \begin{align*} 
        L^\epsilon(\Psi, \Y) = \sup_{\substack{\y \in \mathcal{Y}, \z \in \R^m \\ ||\z -\y||\leq \epsilon}} \frac{||\Psi(\z) - \Psi(\y) ||}{||\z - \y||}.
    \end{align*}
    If $L^\epsilon(\Psi, \Y) < \infty$ for some $\epsilon > 0$, then $\Psi$ is said to be locally Lipschitz continuous.
\end{definition}

Focusing on our problem \eqref{eq:forward_problem}, we remark we are interested in the cases where $\Y = Rg(\A, \X)$. In this case, $\y \in \Y$ implies that $\exists \: \x^{gt} \in \X$ such that $\y = \A\x^{gt}$ and each $\z \in \R^m$ with $|| \z - \y || \leq \epsilon$ can be characterized by $\z = \A\x^{gt} + \e$ for some $\e \in \R^m$ with $||\e|| \leq \epsilon$.
Thus, the definition of $L^\epsilon(\Psi, \Y)$ can be rewritten as:

\begin{align*}
    L^\epsilon(\Psi, \Y) = \sup_{\substack{\x^{gt} \in \X \\ ||\e|| \leq \epsilon}} \frac{|| \Psi(\A\x^{gt} + \e) - \Psi(\A\x^{gt}) ||}{|| \e ||}.
\end{align*}
The importance of the local Lipschitz constant $L^\epsilon(\Psi, \Y)$ lies in its strong relationship to the stability constant $C^\epsilon_\Psi$ of the reconstructor. Indeed, if $\y = \A\x^{gt} \in \Y$ is corrupted by additional noise $\e$ with $||\e||\leq\epsilon$, then $L^\epsilon(\Psi, \Y)$ represents the maximum possible variation of the reconstruction obtained by $\Psi$ around the corrupted $\y$, as stated by the following proposition. \\

\begin{proposition}\label{thm:lip_implies_stab}
    If $\Psi \in \mathcal{R}_\eta$ has local Lipschitz constant $L^\epsilon(\Psi, \Y)$, then, for any $||\e||\leq \epsilon$, it holds:
    \begin{align*}
        || \Psi(\A\x^{gt} + \e) - \x^{gt} || \leq \eta + L^\epsilon(\Psi, \Y) ||\e||\,.
    \end{align*}
\end{proposition}

\begin{proof}
    By the triangle inequality, it follows that:
    \begin{align*}
        || \Psi(\A\x^{gt} + \e) - \x^{gt} || \leq || \Psi(\A\x^{gt} + \e) - \Psi(\A\x^{gt}) || + || \Psi(\A\x^{gt}) - \x^{gt} ||.
    \end{align*}
    Since $||\e||\leq \epsilon$, the definition of local Lipschitz constant implies that:
    \begin{align*}
        || \Psi(\A\x^{gt} + \e) - \Psi(\A\x^{gt}) || \leq L^\epsilon(\Psi, \Y) || \A\x^{gt} + \e - \A\x^{gt} || = L^\epsilon(\Psi, \Y) || \e ||,
    \end{align*}
    whereas the accuracy of $\Psi$ gives:
    \begin{align*}
        || \Psi(\A\x^{gt}) - \x^{gt} || \leq \eta.
    \end{align*}
    Thus, we can conclude:
    \begin{align*}
        || \Psi(\A\x^{gt} + \e) - \x^{gt} || \leq L^\epsilon(\Psi, \Y) ||\e|| + \eta.
    \end{align*}
\end{proof} 

\begin{corollary}\label{cor:stability_constant_lip_constant}
    Under the assumptions of Theorem \ref{thm:lip_implies_stab}, it holds:
    \begin{align*}
        C^\epsilon_\Psi \leq L^\epsilon(\Psi, \mathcal{Y}).
    \end{align*}
\end{corollary} 

\begin{proof}
    From the inequality in Theorem \ref{thm:lip_implies_stab}, we have:
    $$
        || \Psi(\A\x^{gt} + \e) - \x^{gt} || \leq \eta + L^\epsilon(\Psi, \Y) ||\e|| \iff L^\epsilon(\Psi, \Y) \geq \frac{|| \Psi(\A\x^{gt} + \e) - \x^{gt} || - \eta}{|| \e ||}
    $$ 
    for any $\x^{gt} \in \X$ and any $\e \in \R^m$ with $|| \e || \leq \epsilon$. Consequently, $L^\epsilon(\Psi, \mathcal{Y})$ is a majorant of the set:
    $$
        \left\{ \frac{|| \Psi(\A\x^{gt} + \e) - \x^{gt} || - \eta}{|| \e ||}; \x^{gt} \in \X, || \e || \leq \epsilon \right\}.
    $$
    Since $C^\epsilon_\Psi$ is defined as the supremum of this set, by the minimality of the supremum we have $C^\epsilon_\Psi \leq L^\epsilon(\Psi, \mathcal{Y})$.
\end{proof}

We remark that Corollary \ref{cor:stability_constant_lip_constant} proves that $\Psi$ is $\epsilon$-stable if $L^\epsilon(\Psi, \Y) < 1$, yielding a useful sufficient condition to the assessment of stability. 

\begin{example}\label{ex:Tikh}
Under suitable parameter choices, the Tikhonov reconstructor is $\epsilon$-stable for any $\epsilon>0$. 
The Tikhonov reconstructor is built on the Tikhonov method \cite{variational_methods_in_imaging,tikhonov_numerical_methods} and defined as:
\begin{equation}\label{eq:tik_reg_reconstructor}
        \Psi^{\lambda, \boldsymbol{L}}(\y^\delta) = \arg\min_{\x \in \R^{n}} \frac{1}{2} || \A\x - \y^\delta ||^2 + \frac{\lambda}{2} || \boldsymbol{L}\x ||^2,
\end{equation}
where $\lambda>0$ is the regularization parameter and $\boldsymbol{L} \in \R^{d \times n}$ is a matrix such that $\ker (\A) \cap \ker (\boldsymbol{L}) = \{ \boldsymbol{0} \}$. $\boldsymbol{L}$ is usually chosen as the identity or the forward-difference operator. 
We can prove the following proposition regarding Tikhonov stability.
\begin{proposition}\label{prop:tik_is_stable}
    Let $\epsilon>0$ and $\boldsymbol{L} \in \R^{d \times n}$. Then $\exists \> \lambda > 0$ such that:
    \begin{align*}
        L^\epsilon(\Psi^{\lambda, \boldsymbol{L}}, \Y) < 1.
    \end{align*}
\end{proposition}

\begin{proof}
    For any $\lambda > 0$ and any $\y^\delta \in \Y^\epsilon$, it can be shown, by considering the normal equations of \eqref{eq:tik_reg_reconstructor}, that:
    \begin{align*}
        \Psi^{\lambda, \boldsymbol{L}}(\y^\delta) = \Bigl(\A^* \A + \lambda \boldsymbol{L}^* \boldsymbol{L}\Bigr)^{-1} \A^* \y^\delta = \frac{1}{\lambda} \Bigl(\frac{1}{\lambda}\A^* \A + \boldsymbol{L}^* \boldsymbol{L}\Bigr)^{-1} \A^* \y^\delta\,.
    \end{align*}
    Consequently, for any $\y^\delta \in \Y^\epsilon$, it holds that
    $ 
        \Psi^{\lambda, \boldsymbol{L}}(\y^\delta) \to 0$ for  
        $\lambda \to \infty.
    $ 
    Then:
    \begin{align*}
        L^\epsilon(\Psi^{\lambda, \boldsymbol{L}}, \Y) = \sup_{\substack{\x^{gt} \in \X \\ ||\e|| \leq \epsilon}} \frac{|| \Psi^{\lambda, \boldsymbol{L}}(\A\x^{gt} + \e) - \Psi^{\lambda, \boldsymbol{L}}(\A\x^{gt}) ||}{|| \e ||} \to 0 \quad \text{for} \quad \lambda \to \infty,
    \end{align*}
    which implies that, for all $\alpha > 0$, there exists $\bar{\lambda} > 0$ such that for any $\lambda > \bar{\lambda}$, $L^\epsilon(\Psi^{\lambda, \boldsymbol{L}}, \Y) < \alpha$. Choosing $\alpha = 1$ we obtain the required result.
\end{proof}

Corollary \ref{cor:stability_constant_lip_constant} and Proposition \ref{prop:tik_is_stable} demonstrate that it is always possible to build a stable Tikhonov reconstructor. Such property will play a crucial role in Subsection \ref{ssec:ReNN} where we will explain our proposed ReNN approach. 
\end{example}




\section{Stabilizers in the solution of linear inverse problems}\label{sec:stabilizers}

In this section, we delve into additional properties pertinent to stable reconstructors, by introducing the novel concept of {\em stabilizer} which will be exploited in Subsection \ref{ssec:StNN} to define our StNN and StReNN approaches.

\subsection{Stabilizers and properties}\label{subsec:stabilizers}


\begin{definition}\label{def:stabilizers}
    
    A continuous functions $\phi: \R^{m} \to \R^{t}$ is an $\epsilon$-stabilizer of a   reconstructor $\Psi: \R^{m} \to \R^{n}$ if:
    \begin{enumerate}
        \item $\forall \: \e \in \R^{m}$ with $||\e|| \leq \epsilon$, $\exists\ C^\epsilon_\phi \in [0, 1)$ and  $\exists\ \e' \in \R^{n}$ with $||\e'|| = C^\epsilon_\phi ||\e||$ such that:
        \begin{equation*}\label{eq:stabilizer_constant}
            \phi(\A\x + \e) = \phi(\A\x) + \e'.
        \end{equation*}
        
        \item $\exists\ \gamma: \R^{t} \to \R^{n}$ such that $\Psi = \gamma \circ \phi$. 
    \end{enumerate}
    The smallest constant $C^\epsilon_\phi$ for which the definition holds is defined as the stability constant of the stabilizer $\phi$. \\
    We  also define the set:
    \begin{align*}
        \mathcal{S}_\eta^\epsilon = \{ \Psi \in \mathcal{R}_\eta; \: \exists \gamma: \R^t \to \R^n, \exists \: \phi \: \epsilon\text{-stabilizer,} \text{ s.t. } \Psi = \gamma \circ \phi \}.
    \end{align*}
    Whenever $t=m$ and $\gamma: \R^{m} \to \R^{n}$ is a reconstructor, the reconstructor $\Psi$ is said to be $\epsilon$-stabilized with respect to $\gamma$. 
\end{definition}

Note that, in the definition of $\epsilon$-stabilizer, we only require a stability condition for $\phi$ in the first item. 
Interestingly, given a reconstructor $\Psi = \gamma \circ \phi$, we can estimate its $\epsilon$-stability constant $C_\Psi^\epsilon$  by means of the constant $C_\phi^\epsilon$  and the local Lipschitz constant of $\gamma$, as proved in the following proposition.

\begin{proposition}\label{thm:stability_constant_stabilized_reconstructor}
    Let $\Psi: \R^m \to \R^n$, $\Psi = \gamma \circ \phi$, with $\phi$ being an $\epsilon$-stabilizer. If $C_\phi^\epsilon$ is the constant mentioned in Definition \ref{def:stabilizers}, $L^\epsilon(\gamma, \mathcal{T})$ is the local Lipschitz constant of $\gamma$ with $\mathcal{T} = \phi(\Y)$, it holds:
    \begin{equation*}
        C^\epsilon_\Psi \leq L^\epsilon(\gamma, \mathcal{T}) C_\phi^\epsilon.
    \end{equation*}
\end{proposition} 

\begin{proof}
    Let $\x^{gt} \in \X$ and $||\e|| \leq \epsilon$. Then:
    $$ 
        || \Psi(\A\x^{gt} + \e) - \x^{gt} || = || \gamma(\phi(\A\x^{gt} + \e)) - \x^{gt} ||.
    $$ 
    Since $\phi$ is a stabilizer, $\phi(\A\x^{gt} + \e) = \phi(\A\x^{gt}) + \e'$ with $||\e'|| \leq C_{\phi}^\epsilon ||\e||$. Thus:
    \begin{equation*}
    \begin{split}       
        || (\gamma(\phi(\A\x^{gt} + \e)) - \x^{gt} || &= ||\gamma(\phi(\A\x^{gt}) + \e') - \x^{gt} || \\ &\leq || \gamma(\phi(\A\x^{gt}) + \e') - \gamma(\phi(\A\x^{gt})) || + ||\gamma(\phi(\A\x^{gt})) - \x^{gt}|| \\ &\leq \eta + L^\epsilon(\gamma, \mathcal{T}) ||\e'|| = \eta + L^\epsilon(\gamma, \mathcal{T}) C_\phi^\epsilon ||\e||,
    \end{split}
    \end{equation*} 
    which implies that $L^\epsilon(\gamma, \mathcal{T})  C_\phi^\epsilon$ is a majorant of the set:
    $$
        \left\{ \frac{|| \Psi(\A\x^{gt} + \e) - \x^{gt} || - \eta}{|| \e ||}; \: \x^{gt} \in \X, || \e || \leq \epsilon \right\}.
    $$
    Since $C^\epsilon_\Psi$ is defined as the supremum of the same set, by the minimality of the supremum we have $C^\epsilon_\Psi \leq L^\epsilon(\gamma, \mathcal{T}) C_\phi^\epsilon$.
\end{proof}

Theorem \ref{thm:stability_constant_stabilized_reconstructor} implies the following important result.

\begin{theorem}\label{thm:stabilizers_stabilize}
    For any $\epsilon>0$, $\eta_1, \eta_2 > 0$, let $\Psi_1 = \gamma_1 \circ \phi_1 \in \mathcal{S}_{\eta_1}^\epsilon$, and $\Psi_2 \in \mathcal{R}_{\eta_2}$. If:
    \begin{align}\label{eq:stability_constant_stabilizer_inequality}
        C^\epsilon_{\phi_1} \in \left[0, \frac{C^\epsilon_{\Psi_2}}{L^\epsilon(\gamma_1, \mathcal{T})} \right],
    \end{align}
    then:
 $$
     C_{\Psi_1}^\epsilon \leq C_{\Psi_2}^\epsilon.
$$    
\end{theorem}

\begin{proof}
    Since \eqref{eq:stability_constant_stabilizer_inequality} holds by hypothesis 
    and $C^\epsilon_{\Psi_1} \leq L^\epsilon(\gamma_1, \mathcal{T}) C^\epsilon_{\phi_1}$ for $\Psi_1 \in \mathcal{S}_{\eta_1}^\epsilon$ by Theorem \ref{thm:stability_constant_stabilized_reconstructor}, we get:
    
    \begin{equation*}
        C^\epsilon_{\Psi_1} \leq L^\epsilon(\gamma_1, \mathcal{T}) C^\epsilon_{\phi_1} \leq L^\epsilon(\gamma_1, \mathcal{T}) \frac{C^\epsilon_{\Psi_2}}{L^\epsilon(\gamma_1, \mathcal{T})} = C^\epsilon_{\Psi_2},  
    \end{equation*}
    which concludes the proof.
\end{proof}

The theorem yields interesting consequences for the special case where $\Psi_1$ and $\Psi_2$ share the same accuracy. 
For instance, when $\Psi_1 = \gamma_1 \circ \phi_1 \in \mathcal{S}_{\eta}^\epsilon $ and $\Psi_2 \in \mathcal{R}_\eta$, if \eqref{eq:stability_constant_stabilizer_inequality} holds, the theorem suggests that $\Psi_1$ is preferable to $\Psi_2$,  as $\Psi_1$ is more stable than $\Psi_2$.
In addition, we can state the following result, whose  proof is trivial.
\begin{corollary}
    Let $\Psi_1 = \gamma_1 \circ \phi_1 \in \mathcal{S}_{\eta}^\epsilon $ and $\Psi_2 = \gamma_2 \circ \phi_2 \in \mathcal{S}_\eta^\epsilon$. 
If \eqref{eq:stability_constant_stabilizer_inequality} holds, 
then $C^\epsilon_{\phi_1} \leq C^\epsilon_{\phi_2}$. 
\end{corollary}


In the next proposition, we show a result linking the accuracy of a reconstructor $\Psi \in \mathcal{S}_\eta^\epsilon$ to a characterization of its $\epsilon$-stabilizer $\phi$.

\begin{proposition}\label{prop:existence_of_stabilizer}
    Let $\Psi = \gamma \circ \phi \in \mathcal{S}_\eta^\epsilon$. Let:
    \begin{align}\label{eq:definition_approx_injective}
    \sigma(\phi) := \sup \{ || \x_1 - \x_2 ||; \: \x_1, \x_2 \in \X, \phi(\A\x_1) = \phi(\A\x_2) \}\,.
    \end{align}
    Then:
    \begin{equation*}
        \eta^{-1} \leq \frac{2}{\sigma(\phi)}.
    \end{equation*}
\end{proposition}

\begin{proof}
    Let $\x_1, \x_2 \in \X$ such that $\phi(\A\x_1) = \phi(\A\x_2)$. Then:
    \begin{align*}
        || \x_1 - \x_2 || &\leq || \phi(\A\x_1) - \x_1 || + || \phi(\A\x_1) - \x_2 || \\ &= || \phi(\A\x_1) - \x_1 || + || \phi(\A\x_2) - \x_2 || \leq 2 \eta,
    \end{align*}
    which implies that:
    \begin{equation*}\label{eq:estimate_distance_stabilizers}
        \eta \geq \frac{|| \x_1 - \x_2 ||}{2}.
    \end{equation*}
    Since the estimation above holds for any $\x_1, \x_2$ with $\phi(\A\x_1) = \phi(\A\x_2)$, it holds for $\sigma(\phi)$, thus concluding the proof.
\end{proof}

As a consequence of Proposition \ref{prop:existence_of_stabilizer}, if $\phi$ is the constant operator (having $C_\phi^\epsilon=0$ as observed in Example \ref{example:constant_reconstructor}), it gets $\sigma(\phi) = \infty$, which implies that for any $\gamma$, the accuracy of $\Psi = \gamma \circ \phi$ will be zero, whenever $\X$ is unbounded.

Now, in the following proposition, we show that  a sequence of functions $\{ \phi_k \}_{k \in \mathbb{N}}$ approximating  $\Psi\in \mathcal{R}_\eta$, i.e.:
$$
\lim_{k \to \infty} \sup_{\y^\delta \in \Y^\epsilon} || \phi_k(\y^\delta) - \Psi(\y^\delta) || = 0.
$$
can be exploited to construct a good stabilizer.

\begin{proposition}\label{thm:approximators_are_stabilizers}
    Given a reconstructor $\Psi: \R^m \to \R^n$ with local Lipschitz constant $L^\epsilon(\Psi, \Y) < 1$ and a sequence of functions $\{ \phi_k \}_{k \in \mathbb{N}}$ approximating  $\Psi$, there exists $K \in \mathbb{N}$ such that for any $k \geq K$, $C^\epsilon_{\phi_k} <1$.
\end{proposition}

\begin{proof}
    Consider $\x^{gt} \in X$ and $\e \in \R^m$ with $||\e|| \leq \epsilon$. To prove the result, we need to show that:
    $$
        \phi_k(\A\x^{gt} + \e) = \phi_k(\A\x^{gt}) + \e' \ {\mbox for}\ k \geq K,
    $$
    with $||\e'|| = C^\epsilon_\phi ||\e||$ and $C^\epsilon_\phi \in [0, 1)$. \\    
    Let $\e' := \phi_k(\A\x^{gt} + \e) - \phi_k(\A\x^{gt})$, then:
    \begin{align*}
        ||\e'|| &= || \phi_k(\A\x^{gt} + \e) - \phi_k(\A\x^{gt}) || \\ &\leq L^\epsilon(\phi_k, \Y)|| \A\x^{gt} + \e - \A\x^{gt} || = L^\epsilon(\phi_k, \Y)||\e||,
    \end{align*}
    which implies that $C^\epsilon_{\phi_k} \leq L^\epsilon(\phi_k, \Y)$. Since $\{ \phi_k \}_{k \in \mathbb{N}}$ is a sequence of approximators of $\Psi$, for any $k \in \mathbb{N}$ there is a constant $c_k$ such that $|| \phi_k(\y^\delta) - \Psi(\y^\delta) || \leq c_k$ and  $c_k \to 0$ as $k \to \infty$. Consequently, it holds:
    \begin{align*}
    \begin{split}
        &L^\epsilon(\phi_k, \Y) = \sup_{\substack{\x^{gt} \in \X \\ ||\e|| \leq \epsilon}} \frac{|| \phi_k(\A\x^{gt} + \e) - \phi_k(\A\x^{gt}) ||}{|| \e ||} \\ &\leq \sup_{\substack{\x^{gt} \in \X \\ ||\e|| \leq \epsilon}} \frac{|| \phi_k(\A\x^{gt} + \e) - \Psi(\A\x^{gt} + \e) || + || \phi_k(\A\x^{gt}) - \Psi(\A\x^{gt})|| + || \Psi(\A\x^{gt} + \e) - \Psi(\A\x^{gt}) ||}{|| \e ||} \\ &\leq \sup_{\substack{\x^{gt} \in \X \\ ||\e|| \leq \epsilon}} \frac{|| \Psi(\A\x^{gt} + \e) - \Psi(\A\x^{gt}) || + 2c_k}{|| \e ||},
    \end{split}
    \end{align*}
    which implies that $L^\epsilon(\phi_k, \Y) \to L^\epsilon(\Psi, \Y)$ as $k \to \infty$. Since $L^\epsilon(\Psi, \Y) < 1$,  $\exists \: K \in \mathbb{N}$ such that for any $k \geq K$, $L^\epsilon(\phi_k, \Y)<1$. For those values of $k$, $C^\epsilon_{\phi_k} \leq L^\epsilon(\phi_k, \Y)<1$.
\end{proof}


\subsection{Tikhonov stabilizers}\label{ssec:TikStabilizer} 

If we now consider the Tikhonov reconstructor $\Psi=\Psi^{\lambda, \boldsymbol{L}}$ introduced in Example \ref{ex:Tikh}, it is possible to construct a sequence $\{ \phi_k \}_{k \in \mathbb{N}}$ of $\epsilon$-stabilizers. 
In fact, recalling that $L^\epsilon(\Psi, \Y)<1$ for suitable $\lambda > 0$ as stated in Proposition \ref{prop:tik_is_stable}, a simple way to generate the sequence $\{ \phi_k \}_{k \in \mathbb{N}}$ is the following. 
Consider a convergent iterative algorithm for the solution of \eqref{eq:tik_reg_reconstructor}:
$$
    \begin{cases}
    \x^0 \in \R^{n}, \\
    \x^{k+1} = \mathcal{T}_k(\x^k, \y^\delta),
    \end{cases}
$$
where $\mathcal{T}_k(\x^k, \y^\delta)$ models the application of the $k$-th iterate of the algorithm, starting from $\x^k$ and with datum $\y^\delta$.
To set an example, the Conjugate Gradient for Least Squares (CGLS) algorithm is an iterative method solving the normal equations associated with  \eqref{eq:tik_reg_reconstructor}.
Now, for any $k \in \mathbb{N}$ we can define the {\em Tikhonov stabilizers} $\phi_k$ to be the composition of the first $k$ iterations of the algorithm, i.e.:
\begin{equation}\label{eq:from_iterative_to_approximator}
    \phi_k(\cdot) = \bigcircop_{i=1}^k \mathcal{T}_i (\cdot, \y^\delta).
\end{equation}
For the convergence property of the algorithm, $\{ \phi_k \}_{k \in \mathbb{N}}$ is a sequence of functions approximating $\Psi^{\lambda, \boldsymbol{L}}$ and with $C^\epsilon_{\phi_k} <1$ for suitable $k \geq K$. Such property will be fundamental for the stabilization technique we propose in Subsection \ref{ssec:StNN}.

\section{Neural networks for the solution of linear inverse problems}\label{sec:nn_as_reconstructors}
In this section, the theoretical results previously outlined are applied to scenarios where reconstructors are operationalized through neural networks. Concurrently, we delineate our methodologies aimed at advancing current state-of-the-art approaches. Figure \ref{fig:graph_abstract} offers a detailed schematic that encapsulates all the approaches considered within this study. The 'Tik' label refers to the Tikhonov reconstructor $\Psi^{\lambda, \boldsymbol{L}}$, defined in Example \ref{ex:Tikh}.

\begin{figure}
    \centering
    \includegraphics[width=0.8\linewidth]{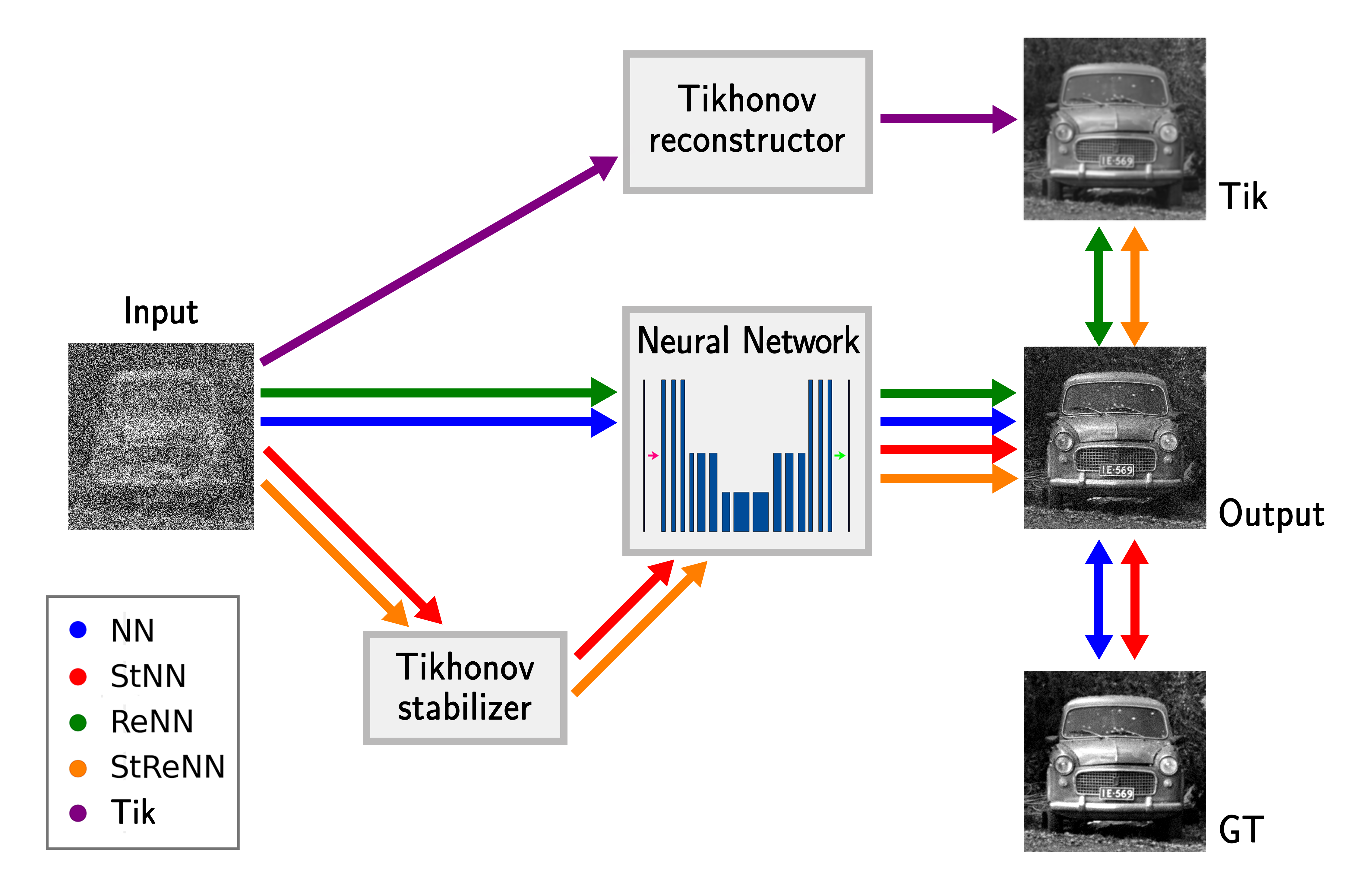}
    \caption{A schematic representation of the proposed methods.}
    \label{fig:graph_abstract}
\end{figure}

\subsection{Parameter-dependent families of reconstructors}
We now consider a family of reconstructors $\{\Psi_\Theta \}_{\Theta \in \R^s}$, depending on a vector of parameters $\Theta$, approximating a reconstructor $\Psi$ to solve problem \eqref{eq:disc_fred_eq}.
We prove in the following theorem that the stability of $\Psi_\Theta$ is strongly related to the stability of  $\Psi$.
\begin{theorem}[Approximation Theorem for Reconstructors]\label{thm:stability_theorem}
 Let $\Psi$ be an $\eta^{-1}$-accurate reconstructor and let $\{ \Psi_\Theta \}_{\Theta \in \R^s}$ be a set of reconstructors with accuracy $\eta_\Theta^{-1}$ for any $\Theta$. We define, for any $\Theta \in \R^s$:
 $$ \Delta(\Theta) := \sup_{\x^{gt} \in \X} || \Psi_\Theta(\A\x^{gt}) - \Psi(\A\x^{gt}) || $$
 and:
 $$\Delta_\epsilon (\Theta):= \sup_{\substack{\x^{gt} \in \X \\ || \e || < \epsilon}} || \Psi_\Theta(\A\x^{gt} + \e) - \Psi(\A\x^{gt} + \e) ||.$$
If $\Delta(\Theta) \to 0$ when $\Theta \to \Theta^*$, then: 
	\begin{equation}
	\lim_{\Delta(\Theta) \to 0}\eta_\Theta =\eta.
         \label{eq:limeta}
	\end{equation}
Moreover, if $\Delta_\epsilon(\Theta) \to 0$ when $\Theta \to \Theta_\epsilon^*$, then:
		\begin{equation}
	\lim_{\Delta_\epsilon(\Theta) \to 0} C^\epsilon_{\Psi_\Theta} = C^\epsilon_\Psi.
        \label{eq:limc}
	\end{equation}
\end{theorem} 

\begin{proof}
	Consider $\x^{gt} \in \X$. Since:
	\begin{align*}
		|| \Psi_\Theta(\A\x^{gt}) - \x^{gt} || \leq || \Psi(\A\x^{gt}) - \x^{gt} || + || \Psi_\Theta(\A\x^{gt}) - \Psi(\A\x^{gt}) ||
	\end{align*}
	and:
	\begin{align*}
		|| \Psi_\Theta(\A\x^{gt}) - \x^{gt} || \geq || \Psi(\A\x^{gt}) - \x^{gt} || - || \Psi_\Theta(\A\x^{gt}) - \Psi(\A\x^{gt}) ||,
	\end{align*}
	it holds that:
	$$
	| \> || \Psi_\Theta(\A\x^{gt}) - \x^{gt} || - || \Psi(\A\x^{gt}) - \x^{gt} || \> | \leq || \Psi_\Theta(\A\x^{gt}) - \Psi(\A\x^{gt}) || \leq \Delta(\Theta),
	$$
	which implies that $|| \Psi_\Theta(\A\x^{gt}) - \x^{gt} || \to || \Psi(\A\x^{gt}) - \x^{gt} ||$ as $\Delta(\Theta) \to 0$ and consequently, $\eta_\Theta \to \eta$ as $\Delta(\Theta) \to 0$. \\
	
	Now, consider $\epsilon>0$ and  $\e \in \R^m$ with $|| \e || \leq \epsilon$. A similar computation shows that:
	\begin{align*}
		| \> || \Psi_\Theta(\A\x^{gt} + \e) - \x^{gt} || - || \Psi(\A\x^{gt} + \e) - \x^{gt} || \> | \leq \Delta_\epsilon(\Theta),
	\end{align*}
	which implies that $ || \Psi_\Theta(\A\x^{gt} + \e) - \x^{gt} || \to || \Psi(\A\x^{gt} + \e) - \x^{gt} ||$ for $\Delta_\epsilon(\Theta) \to 0$. Consequently, for $\Delta_\epsilon(\Theta) \to 0$,
	\begin{align*}
		C^\epsilon_{\Psi_\Theta} = \sup_{\substack{\x^{gt} \in \X \\ || \e || \leq \epsilon}} \frac{|| \Psi_\Theta(\A\x^{gt}+e) -  \x^{gt}|| - \eta_\Theta}{||\e||} \to \sup_{\substack{\x^{gt} \in \X \\ || \e || \leq \epsilon}} \frac{|| \Psi(\A\x^{gt}+e) -  \x^{gt}|| - \eta}{||\e||} = C^\epsilon_\Psi,
	\end{align*}
	which concludes the proof.
\end{proof}

\begin{corollary}\label{cor:estimate_accuracy_nn}
    For any $\Theta \in \R^s$, it holds:
    \begin{equation*}
        \eta_\Theta \leq \eta + \Delta(\Theta).
    \end{equation*}
\end{corollary}

\begin{proof}
    Consider $\x^{gt} \in \X$. Then:
	\begin{align*}
		|| \Psi_\Theta(\A\x^{gt}) - \x^{gt} || \leq || \Psi_\Theta(\A\x^{gt}) -  \Psi(\A\x^{gt})|| + || \Psi(\A\x^{gt}) - \x^{gt} ||.
	\end{align*}
	Since $|| \Psi_\Theta(\A\x^{gt}) -  \Psi(\A\x^{gt})|| \leq \Delta(\Theta)$ by hypothesis and  $|| \Psi(\A\x^{gt}) - \x^{gt} || \leq \eta$ since $\Psi$ is $\eta^{-1}$-accurate, then:
	\begin{equation*}
		|| \Psi_\Theta(\A\x^{gt}) - \x^{gt} || \leq \eta + \Delta(\Theta),
	\end{equation*}
	which shows that $\eta_\Theta \leq \eta + \Delta(\Theta)$.
\end{proof}

Note that $\Delta(\Theta)$ and $\Delta_\epsilon(\Theta)$ are, in general, not independent, as proved in the following proposition.

\begin{proposition}\label{thm:Delta_leq_Delta_epsilon}
    For any $\epsilon>0$, let $\Delta(\Theta)$ and $\Delta_\epsilon(\Theta)$ be the quantities defined in Theorem \ref{thm:stability_theorem}. Then:
    \begin{equation*}
        \Delta(\Theta) \leq \Delta_\epsilon(\Theta).
    \end{equation*}
\end{proposition}
\begin{proof}
    Observe that, by definition of $\Y$ and $\Y^\epsilon$, $\Delta(\Theta)$ and $\Delta_\epsilon(\Theta)$ can be rewritten as:
    \begin{align*}
        &\Delta(\Theta) = \sup_{\y \in \Y} || \Psi_\Theta(\y) - \Psi(\y) ||, \\
        &\Delta_\epsilon(\Theta) = \sup_{\y \in \Y^\epsilon} || \Psi_\Theta(\y) - \Psi(\y) ||, 
    \end{align*}
    where $\Y^\epsilon = \{ \y + \e; \y \in \Y, || \e || \leq \epsilon \} \supseteq \Y$. The result follows from the property that the supremum of a set must be larger than the supremum of its subsets.
\end{proof}

An insight on the stability properties of $\Psi_\Theta$ can be obtained by the following proposition.

\begin{proposition}\label{prop:bound_on_delta_theta}
    Let $\Psi_\Theta$ be a reconstructor parameterized by $\Theta \in \R^s$, approximating a reconstructor $\Psi$ with error $\Delta(\Theta) > 0$. Let $\eta_\Theta^{-1}$ and $\eta^{-1}$ be the accuracy of $\Psi_\Theta$ and $\Psi$, respectively. If:
    \begin{equation}\label{eq:upper_bound_delta_theta}
        \Delta(\Theta) \leq \bar{\eta}(\A, \epsilon, \X) - \eta
    \end{equation}
    for a fixed $\epsilon>0$, where $\bar{\eta}(\A, \epsilon, \X)$ is the constant defined in Corollary \ref{cor:bound_on_acc}, then $C_{\Psi_\Theta}^\epsilon \geq 1$.
\end{proposition}

\begin{proof}
    Let $\epsilon > 0$ be fixed. By Corollary \ref{cor:estimate_accuracy_nn}, the accuracy of $\Psi_\Theta$ can be estimated as $\eta_\Theta \leq \eta + \Delta(\Theta)$. Consequently, by Corollary \ref{cor:bound_on_acc}, if $\eta + \Delta(\Theta) \leq \bar{\eta}(\A, \epsilon, \X)$, then $\eta_\Theta \leq \bar{\eta}(\A, \epsilon, \X)$, which implies that $C_{\Psi_\Theta}^\epsilon \geq 1$.
\end{proof}

In the following paragraphs, we will analyze two particular families of reconstructors $\{\Psi_\Theta \}_{\Theta \in \R^s}$.

\subsection{Neural Networks as reconstructors: the NN approach}\label{ssec:NN}
Now we consider the set of neural networks defined by a fixed architecture as the family  $\{ \Psi_\Theta \}_{\Theta \in \R^s}$.

\begin{definition}
	Given a neural network architecture $\mathcal{A} = (\nu, S)$ where $\nu = (\nu_0, \nu_1, \dots, \nu_L) \in \mathbb{N}^{L+1}$, $\nu_0 =m, \nu_L=n$, defines the width of each layer and $S = (S_{1, 1}, \dots, S_{L, L}), S_{j, k} \in \R^{\nu_j \times \nu_k}$ is the set of matrices representing the skip connections, we define the parametric family of neural network reconstructors with architecture $\mathcal{A}$, parameterized by $\Theta \in \R^s$, as
	$$
	\mathcal{F}_\Theta^\mathcal{A} = \{ \Psi_\Theta : \R^{m} \to \R^{n}; \Theta \in \R^s \},
	$$
	where $\Psi_\Theta(\y^\delta) = \z^L$ is given by:
	\begin{align}
		\begin{cases}
			\z^0 = \y \\
			\z^{l+1} = \rho(W^l \z^l + b^l + \sum_{k=1}^l S_{l, k} \z^k) \quad \forall \> l = 0, \dots, L-1 
		\end{cases}
	\end{align}
	and
	$W^l \in \R^{\nu_{l+1} \times \nu_l}$ is the weights matrix, $b^l \in \R^{\nu_{l+1}}$ is the bias vector.
\end{definition}

Given $\mathcal{D} \subseteq \X$, consider the dataset $\D = \{ (\y_i^\delta, \x^{gt}_i); \x^{gt}_i \in \mathcal{D} \}_{i=1}^{N_\D}$ of images according to \eqref{eq:forward_problem}.
Training a neural network to solve the inverse problem \eqref{eq:forward_problem} results in finding the parameters $\Theta^*$ such that the associated reconstructor $\Psi_{\Theta^*} \in \mathcal{F}_\Theta^\mathcal{A}$ satisfies:
\begin{equation}
 \Psi_{\Theta^*} \in   \arg\min_{\Psi_\Theta \in \mathcal{F}_\Theta^\mathcal{A}} \frac{1}{N_\D} \sum_{i=1}^{N_\D} \ell (\Psi_\Theta (\y_i^\delta), \x^{gt}_i),
    \label{eq:nnmin}
\end{equation}
where  $\delta \geq 0$ and $\ell: \R^{n} \times \R^{n} \to \R_+$ is the loss function. \\
In this work, we consider as reconstructors $\Psi_\Theta$ the neural networks trained with the Mean Squared Error (MSE) loss.  We will name this family as NN, in the following.
We first apply NN onto noiseless data ($\delta=0$), thereby \eqref{eq:nnmin} corresponds to:
  \begin{equation}\label{eq:nn_training}
    \min_{\Psi_\Theta \in \mathcal{F}_\Theta^\mathcal{A}} \sum_{i=1}^{N_\D} || \Psi_\Theta(\y_i) - \x^{gt}_i||_2^2 = \min_{\Psi_\Theta \in \mathcal{F}_\Theta^\mathcal{A}} \sum_{i=1}^{N_\D} || \Psi_\Theta(\A\x^{gt}_i) - \Psi^\dagger(\A\x^{gt}_i)||_2^2,
\end{equation}
which results in the minimization of $\Delta(\Theta)$ as introduced in Theorem \ref{thm:stability_theorem} with $\Psi=\Psi^\dagger$.

We observe that when $\A$ is ill-conditioned, $\bar{\eta}(\A, \epsilon, \X)$ is large. This becomes particularly apparent when  $\X = \R^n$, as under these circumstances, $\bar{\eta}(\A, \epsilon, \X)$  is bounded below by a quantity depending on $C(\A) = \frac{1 - \sigma_n}{\sigma_n}$. Additionally, the value of $\Delta(\Theta^*)$ derived from NN training likely meets the established inequality in Proposition \ref{prop:bound_on_delta_theta}, which leads to instability. This confirms that effective neural network training can produce a very accurate but unstable reconstructor $\Psi_\Theta$.

A widely adopted strategy to bolster robustness in neural networks is known as noise injection. This technique involves adding noise to the input of the network during its training phase. In this context, the set of reconstructors $\Psi_\Theta$, referred to as iNN, is defined by a neural network trained through the following equation:

\begin{equation}\label{eq:nn_training_NI}
\min_{\Psi_\Theta \in \mathcal{F}\Theta^\mathcal{A}} \sum_{i=1}^{N_\D} || \Psi_\Theta(\y_i^\delta) - \x^{gt}_i||_2^2,
\end{equation}
where $\delta > 0$. Research detailed in \cite{training_with_noise_injection} has demonstrated that this approach effectively introduces a Tikhonov regularization term into the loss function. Although this technique, as described in \cite{solving_inverse_problems_data_driven}, enhances the stability of the resultant network, the impact of noise injection on the accuracy of the model remains somewhat ambiguous. Furthermore, the optimal amount of noise to be added to each input to optimize the balance between stability and accuracy is still a subject of investigation.

\subsection{Regularized NN-based reconstructors: the ReNN approach}\label{ssec:ReNN}
To develop a reconstructor with improved stability compared to standard neural networks (NN), we harness the properties of Tikhonov regularization. It is important to note that a Tikhonov regularized reconstructor $\Psi^{\lambda, \boldsymbol{L}}$ achieves stability for an appropriately chosen regularization parameter, as delineated in Proposition \ref{prop:tik_is_stable}. This methodology will be referred to as the Regularized Neural Network (ReNN), denoted as $\Psi_\Theta^{\lambda, \boldsymbol{L}}$.
ReNN is defined by training a neural network with a new loss $\ell$ as:
\begin{equation}\label{eq:renn_training}
  \Psi_\Theta^{\lambda, \boldsymbol{L}} \in  \arg\min_{\Psi_\Theta \in \mathcal{F}_\Theta^\mathcal{A}} \sum_{i=1}^{N_\D} || \Psi_\Theta(\y^\delta_i) - \Psi^{\lambda, \boldsymbol{L}}(\y^\delta_i)||_2^2,
\end{equation}
with $\delta > 0$. We underline that  ReNN does not require any ground-truth solutions $\x^{gt}$ since the target is computed from the corrupted datum $\y^\delta$ via the Tikhonov-regularized reconstructor.
Furthermore, in the training of ReNN, noise is present not solely to the input of the neural network model, as is the case with iNN, but also to the input of the Tikhonov-regularized reconstructor, which is responsible for generating the target.
In the following, we consider for simplicity the case $\X=\R^n$, but similar results hold for a general $\X \subset \R^n$.

Starting from inequality \eqref{eq:upper_bound_delta_theta} it is easy to notice that \eqref{eq:renn_training} corresponds to the minimization of $\Delta_\epsilon(\Theta)$ in Theorem \ref{thm:stability_theorem}. Moreover, by Theorem \ref{thm:Delta_leq_Delta_epsilon}, if $\Delta_\epsilon(\Theta)$ is small, as it is common when $\Psi_\Theta$ is a neural network, then $\Delta(\Theta) \in [0, \Delta_\epsilon(\Theta)]$ is also small. Regarding the right hand side $\bar{\eta}(\A, \epsilon, \X) - \eta$ of \eqref{eq:upper_bound_delta_theta}, it is noted that in this instance $\eta=\eta(\lambda)$ and $\eta(\lambda) \to \infty$ for $\lambda \to \infty$. Consequently,  for sufficiently large values of $\lambda$, it is probable that ReNN does not fulfill the conditions of  \eqref{eq:upper_bound_delta_theta}.

Moreover, minimizing $\Delta_\epsilon(\Theta)$
is crucial for enforcing the method's stability,  as proven by Theorem \ref{thm:stability_theorem}, where we have shown that in our hypothesis the stability constant $C^\epsilon_{\Psi_\Theta^{\lambda,\boldsymbol{L}}} < 1$ for sufficiently small $\Delta_\epsilon(\Theta)$.
 Hence,  effective training of ReNN should produce an accurate and stable reconstructor. The pseudocode to compute $\Psi^{\lambda, \boldsymbol{L}}_\Theta$ is given in Algorithm \ref{alg:renn_algorithm}.\\

\begin{algorithm}[hbt!]
    \caption{Regularized Neural Network (ReNN)}\label{alg:renn_algorithm}
\begin{algorithmic}
\STATE{{\bf input} a collection $\{ \x^{gt}_i \}_{i=1}^{N_\D} \subseteq \X$ of data points, a noise level $\delta>0$, $\A \in \R^{m \times n}$ and a stable reconstrctor $\Psi^{\lambda, \boldsymbol{L}}$\;}
    
\FOR{$i \gets 1:N_\D$}
\STATE{Sample $\e_i \sim \mathcal{N}(\boldsymbol{0}, \delta^2\I)$}
\STATE{Compute $\y^\delta_i \gets \A\x^{gt}_i + \e_i$\;}
\ENDFOR 
\STATE{
    Solve
    $$
        \min_{\Psi_\Theta \in \mathcal{F}_\Theta^\mathcal{A}} \sum_{i=1}^{N_\D} || \Psi_\Theta(\y^\delta_i) - \Psi^{\lambda, \boldsymbol{L}}(\y^\delta_i) ||_2^2.
    $$}  
\RETURN a trained ReNN $\Psi_\Theta$
\end{algorithmic}    
\end{algorithm}

\subsection{Stabilization on NN and ReNN: St- approaches} \label{ssec:StNN}
In the remainder of this section, we discuss an application of the stabilizers, introduced in Section \ref{sec:stabilizers}, to improve the stability of neural network-based reconstructors.
We propose new reconstructors $\Psi \in \mathcal{S}_\eta^\epsilon$, $\Psi=\gamma \circ \phi$ where $\gamma$ is a neural network based reconstructor. 
In particular, we consider $\phi$ as the Tikhonov $\epsilon$-stabilizer $\phi_k$ defined in Subsection \ref{ssec:TikStabilizer} and obtained by $k$ iterations of the CGLS algorithm, with a suitable $k$. 
When $\gamma$ is chosen as NN, iNN, ReNN we obtain the $\epsilon$-stabilized reconstructors StNN, StiNN, and StReNN, respectively.

Note that, in this case, we can apply Theorem \ref{thm:stabilizers_stabilize} with $\Psi_1 = \gamma \circ \phi_k$ and $\Psi_2 = \gamma$, and whenever we choose $\phi_k$ such that: 
\begin{align}\label{eq:stabilization_condition}
    C_{\phi_k}^\epsilon \leq \frac{C_\gamma^\epsilon}{L^\epsilon(\gamma, \Y)},
\end{align}
the $\epsilon$-stabilized reconstructor $\Psi_1$ gets more stable than its unstabilized version  $\Psi_2$. 
We remark that it is always possible to find a Tikhonov stabilizer $\phi_k$ fitting \eqref{eq:stabilization_condition}, by suitably tuning $\lambda$ and $k$. 
Clearly, 
this comes at the expense of accuracy as discussed in Proposition \ref{prop:existence_of_stabilizer}, but we will show  that the accuracy does not suffer excessively, as evidenced by empirical results in Section \ref{sec:results}.

\section{Experimental setup}\label{sec:expsetup}
To assess the theoretical issues proposed, we conducted a series of experiments. It is important to highlight that all tests were carried out utilizing the same end-to-end U-net architecture.
For details on the architecture and its training, you can refer to  \cite{green_post_processing, evangelista2023rising}. 
In the following experiments, the stabilizer applied to all the considered reconstructors is obtained with  $k=3$ iterations of the CGLS algorithm on \eqref{eq:tik_reg_reconstructor}. 
The codes can be found in our GitHub repository at \url{https://github.com/loibo/ToBeOrNotToBeStable}. 

As a test case, we consider image deblurring \cite{hansen_deblurring_images}, a common inverse problem in imaging. In this case, $\A$ is a block circulant matrix with circulant blocks obtained from a convolutional kernel with periodic boundary conditions \cite{hansen_deblurring_images}. In our experiments, we use the $11 \times 11$ Gaussian blur filter $\mathcal{K}$:
\begin{equation}\label{eq:convolution_kernel}
    \mathcal{K}_{i, j} = e^{- \frac{1}{2} \frac{i^2 + j^2}{\sigma_G^2}}, \quad  i, j \in \{-5, \dots, 5\}    
\end{equation}
with variance $\sigma_G^2 = 1.3$.

\subsection{Dataset}\label{sec:dataset}
Our results have been tested on the famous GoPro image dataset (\url{https://seungjunnah.github.io/Datasets/gopro}), introduced in \cite{gopro_dataset}, which is constituted by high-resolution RGB images. All the images have been cropped into patches of size $256 \times 256$ (without overlapping), converted into grayscale,  normalized in $[0,1]$, and labeled as $\x^{gt}_i, i=1, \ldots {N_\D}$ with ${N_\D}=3614$. 
We generated the blurred and noisy data $\y_i^\delta = \A \x^{gt}_i+\e$, where $\e \sim \mathcal{N}(\mathbf{0}, \delta^2 \I)$.
We need the following data sets to train the three considered neural networks-based reconstructors.

\begin{itemize}
    \item For the NN training (see \eqref{eq:nn_training}) we consider the set $\D = \{(\y_i, \x^{gt}_i)\}_{i=1}^{N_\D}$ containing the couples of images constituted by the blurred noiseless datum  (i.e. $\delta=0$) and the exact $\x^{gt}_i$ target picture. 
    \item For the iNN training (see \eqref{eq:nn_training_NI}) we consider the set $\D_{\delta} = \{(\y^\delta_i, \x^{gt}_i)\}_{i=1}^{N_\D}$ containing the couples of images constituted by the blurred and noisy datum  $ \y^{\delta}_i$ and the exact $\x^{gt}_i$ target picture. 
    \item For the ReNN training (see \eqref{eq:renn_training}) we consider the set $\D^{\lambda, \boldsymbol{L}}_{\delta} = \{(\y^{\delta}_i, \Psi^{\lambda, \boldsymbol{L}}( \y^\delta_i)) \}_{i=1}^{N_\D}$ containing the couples of images constituted by the blurred and noisy datum $ \y^{\delta}_i$ and the target image computed by the Tikhonov reconstructor (using $\boldsymbol{L} = \I$ in \eqref{eq:tik_reg_reconstructor}). In particular, we choose $\lambda$ heuristically and we computed $\Psi^{\lambda, \boldsymbol{L}}(\y^\delta)$ by means of the CGLS algorithm \cite{hansen1998rank} to solve the normal equations of \eqref{eq:tik_reg_reconstructor}.
\end{itemize}
We finally split the $N_\D$ data samples into train and test subsets, with $N_{train} = 2503$ and $N_{test} = 1111$.

\subsection{Results evaluation}

In order to estimate in our experiments the accuracy and the stability constants of a given reconstructor $\Psi$ we compute the {\em empirical} accuracy $\hat{\eta}^{-1}$ and the {\em empirical} stability constant $\hat{C}^\epsilon_\Psi$, over the test set $\mathcal{TS}$.
They are respectively defined as:

\begin{equation}\label{eq:empirical_accuracy}
    \hat{\eta} = \sup_{\x^{gt} \in \mathcal{TS}} || \Psi(\A\x^{gt}) - \x^{gt} ||
\end{equation}
and:
\begin{equation}\label{eq:empirical_stability}
    \hat{C}^\epsilon_\Psi = \sup_{\x^{gt} \in \mathcal{TS}} \frac{|| \Psi(\A\x^{gt} + \e) - \x^{gt} || - \hat{\eta}}{|| \e ||}\,,
\end{equation}
where  $\e \sim \mathcal{N}(\mathbf{0}, \delta^2 \I$)  differs for each datum $\x^{gt} \in \mathcal{TS}$). 
Finally, we compute the  {\em empirical reconstruction error} on the test set as:
\begin{equation*}\label{eq:reconst_error}
   \mathcal{E}(\Psi, \delta) = \sup_{\x^{gt} \in \mathcal{TS}} ||\Psi(\y^\delta)-\x^{gt} ||.
\end{equation*}
To evaluate a single image reconstruction, we also compute the widely used Structural Similarity Index (SSIM) \cite{SSIM_definition}, taking values in $[0,1]$.

To augment the stochastic nature of our experiments, we replicated the tests on the test set $T=20$ times, each with different realizations of noise. In the following, we report the maximum value of the computed parameters $\hat\eta$ and $\hat{C}^\epsilon_\Psi$ over the $T$ experiments.

\section{Numerical Results}\label{sec:results}

In this section, we present the outcomes achieved in terms of empirical accuracy, stability, and reconstruction error for the solvers proposed in this study. 
The objective of this section is twofold: firstly, to validate the key theoretical findings established in the previous part of the paper, with a particular emphasis on the deep learning-based reconstructors introduced in Section \ref{sec:nn_as_reconstructors}; and secondly, to examine the impact of the stabilizer in scenarios where the noise levels exceed those the parameters were initially selected for.

\subsection{Results with NN-based reconstructors}\label{sec:experiment_NN}

\begin{table}\sffamily
\centering
\begin{tabular}{ccc}
$\x^{gt}$ &  NN  & StNN \\
& \textit{(SSIM = 0.9864)} & \textit{(SSIM = 0.9142)} \\
\includegraphics[width=0.28\textwidth,trim=50 10 20 60,clip]{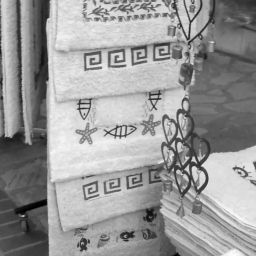}
& \includegraphics[width=0.28\textwidth,trim=50 10 20 60,clip]{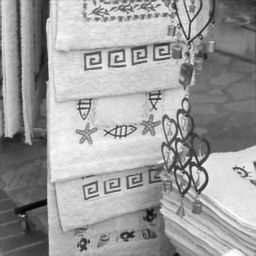}
& \includegraphics[width=0.28\textwidth,trim=50 10 20 60,clip]{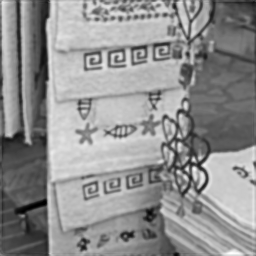}  \vspace{2mm}\\ 
$\y^\delta$ & NN & StNN \\
\textit{(SSIM = 0.8171)} & \textit{(SSIM = 0.0647)} &\textit{(SSIM = 0.8301)} \\
\includegraphics[width=0.28\textwidth,trim=50 10 20 60,clip]{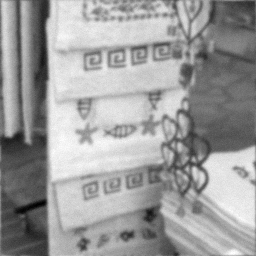} 
& \includegraphics[width=0.28\textwidth,trim=50 10 20 60,clip]{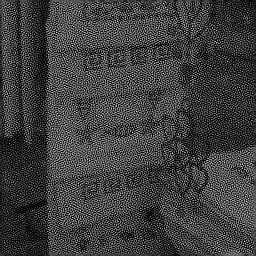} 
& \includegraphics[width=0.28\textwidth,trim=50 10 20 60,clip]{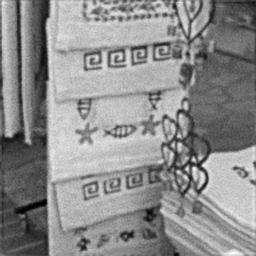}  \\
\end{tabular}
\captionof{figure}{Results obtained by the NN and StNN reconstructors on a single test image $y^{\delta}$ with $delta=0$ (first row) and $\delta=0.01$ (second row). The ground truth clean image is also reported for reference.}
\label{fig:noiseless_images}
\end{table}

We begin by considering the NN and iNN approaches, and their stabilized counterparts, StNN and StiNN, assuming the availability of ground truth images $\x^{gt}_i, i=1, \dots {N_\D}$.

The first experiment concerns NN and  StNN. The first row of Figure \ref{fig:noiseless_images} shows the reconstructions obtained with both the methods on one image $\y_i$ from the test set (without noise added).
To assess the stability of our frameworks concerning unseen noise on the data, we also tested the NN reconstructor on noisy images $\y_i^\delta=\y_i + \e_i$ with $\e_i \sim \mathcal{N}(\mathbf{0}, \delta^2 \I)$ and $\delta = 0.01$. The second row of Figure \ref{fig:noiseless_images} displays the reconstructions obtained on the same test image. 

\begin{tabular}{cc}
\begin{minipage}[b]{0.42\textwidth}
\centering
\begin{tabular}{l cc}
    \toprule
         & NN & StNN \\
        \midrule
        $\hat{\eta}^{-1}$ & 0.1203 & 0.0616 \\
        $\hat{C}^\epsilon_\Psi (\delta = 0.01)$ & 36.7298 & 0.1579 \\
        \bottomrule \vspace{5mm}
    \end{tabular}
    \captionof{table}{Values of empirical accuracy and $\epsilon$-stability constant obtained by NN and StNN reconstructors, trained with  $\delta = 0$.}\label{tab:table_NN}
\end{minipage}
&
\begin{minipage}[b]{0.42\textwidth}
    \centering
    \includegraphics[width=\linewidth]{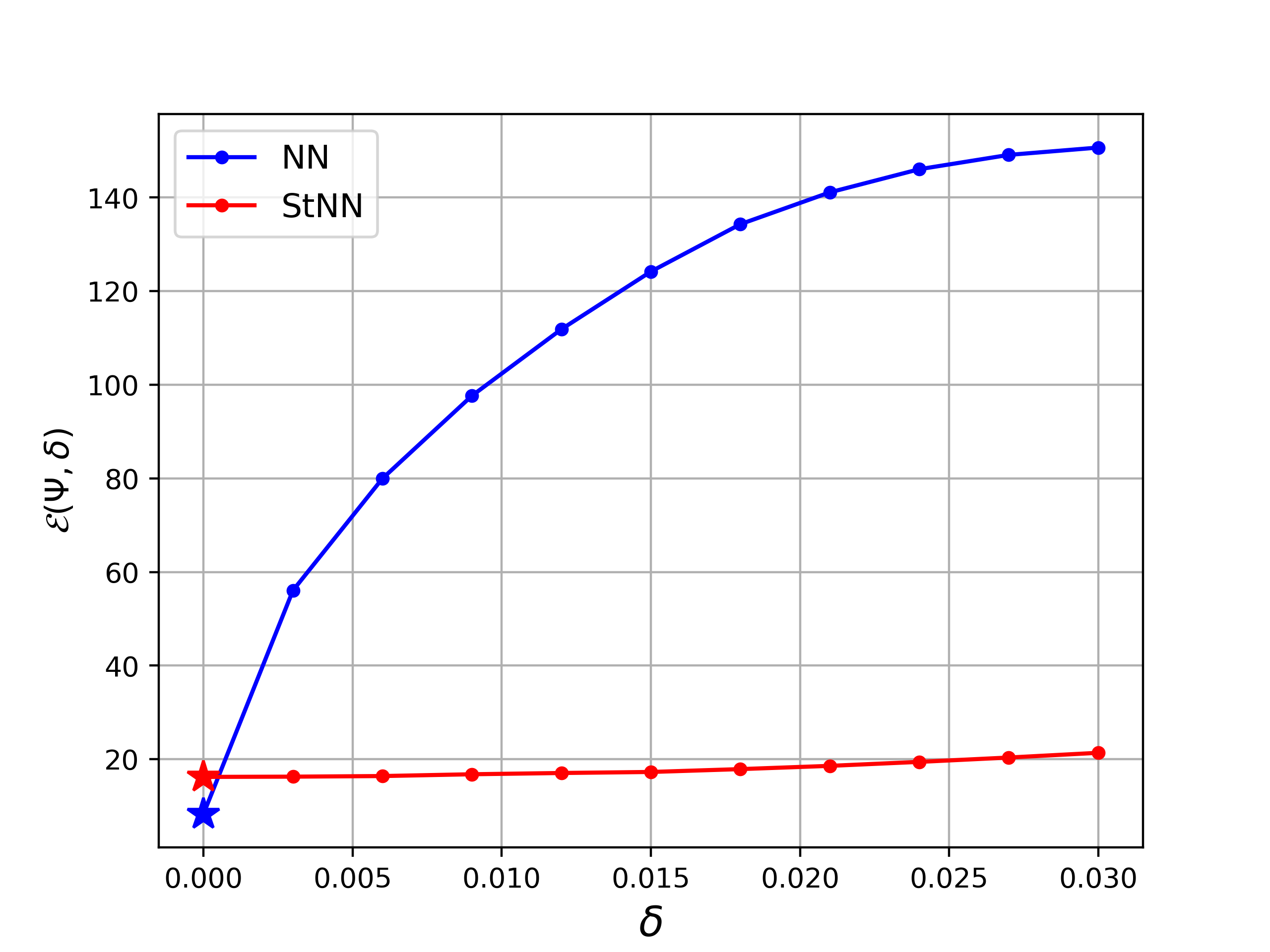}
    \captionof{figure}{Plots of the empirical error yielded by NN and StNN reconstructors for increasing values of $\delta$ in the test images.}\label{fig:results_ex_A}
\end{minipage}
\end{tabular}

From the images presented in Figure \ref{fig:noiseless_images} and their SSIM values, it is observable that the NN reconstructor excels in restoring the blurred image, yet it demonstrates its unreliability as soon as even a minimal amount of noise is added to the data. In contrast, StNN emerges as an effective compromise between accuracy, as evidenced by the high-quality image in the first row with noise-free data, and stability, highlighted by the superior quality of the StNN image compared to the NN one in the second row under noisy conditions.
The Table \ref{tab:table_NN} reports the values of the empirical accuracy $\hat\eta^{-1}$ and empirical stability constant $\hat{C}^\epsilon_\Psi$ for the considered methods on the whole test set. It confirms that there is a trade-off between accuracy and stability, as proved in Theorem \ref{thm:stab_vs_acc_tradeoff}, and that the stabilization strategy improves the value of $\hat{C}^\epsilon_\Psi$ for NN.

To further investigate the different behavior of the two reconstructors for increasing values of $\delta$, in Figure \ref{fig:results_ex_A} we plot the reconstruction error for $\delta \in [0,0.03]$. The value of $\delta=0$ used in the training is indicated with a star marker.
We note that the StNN curve is characterized by a notably flat trajectory, in contrast to the NN curve which exhibits a rapid increase. This observation aligns with and reinforces the insights gathered from previous analyses.

In the second experimental setting, we considered the iNN reconstructor, trained by \eqref{eq:nn_training_NI},  with $\delta = 0.025$.
Table \ref{tab:table_iNN} reports the empirical accuracy and stability computed for both iNN and its stabilized version, StiNN, when the methods are tested on data $\y_i^\delta$ with $\delta=0.025, 0.060, 0.125$, respectively. The table shows that injecting noise in the observed data during training produces slightly less accurate but far more stable reconstructors (as visible by comparing the results with unseen noise in Table \ref{tab:table_iNN} to those in Table \ref{tab:table_NN}).

\begin{table}
    \centering
    \begin{tabular}{l cc}
    \toprule
                     & iNN & StiNN  \\
     \midrule
        $\hat{\eta}^{-1}$ & 0.0707 & 0.0606  \\
        $\hat{C^\epsilon_\Psi} (\delta = 0.025)$ & 0.0899 & 0.0703  \\
        $\hat{C^\epsilon_\Psi} (\delta = 0.060)$ & 0.4309 & 0.2122 \\
        $\hat{C^\epsilon_\Psi} (\delta = 0.125)$ & 0.8385 & 0.6215  \\
    \bottomrule
    \end{tabular}
    \caption{Values of empirical accuracy and stability constant obtained for iNN and StiNN reconstructors, trained  on noisy data with $\delta = 0.025$ and tested with different values of $\delta$.} 
    \label{tab:table_iNN}
\end{table}

\subsection{Results with ReNN-based reconstructors}\label{sec:experiment_ReNN}

In this subsection, we focus on the application of the proposed ReNN reconstructor and its stabilized variant StReNN on noisy data characterized by $\delta=0.025$. It is important to recall that ReNN is trained following the methodology outlined in \eqref{eq:renn_training} and utilizes a dataset that does not include the exact $\x^{gt}$  images.

The target images are the output of Tikhonov reconstructor $\Psi^{\lambda, \boldsymbol{L}}$ applied to the data $\y^\delta_i, i=1 \ldots N_\D $. The Tikhonov regularization parameter $\lambda$ has been heuristically chosen as $\lambda=0.31$ to obtain a  small reconstruction error on the  training set.
The methods have been tested on noisy data with $\delta=0.025, 0.060, 0.125$, respectively. 
The outcomes obtained in terms of accuracy and stability constants are reported in Table \ref{tab:table_ReNN}. In the final column of this table, we also include the metrics pertaining to the Tikhonov reconstructor. It is observed that the accuracy of the three methods is quite comparable. 
Notably, the stability of the regularized NN-based reconstructors surpasses that of the Tikhonov method. Furthermore, the application of stabilization to ReNN exhibits increasingly beneficial effects as the noise level in the data escalates, as evidenced in the table's last row.

\begin{table}
    \centering
    \begin{tabular}{l ccc}
    \toprule
                    & ReNN & StReNN & Tik \\
     \midrule
        $\hat{\eta}^{-1}$  & 0.0461 & 0.0420 & 0.0474 \\
        {$\hat{C^\epsilon_\Psi} (\delta = 0.025)$}& 0.0270 & 0.0150 & 0.0614 \\
        {$\hat{C^\epsilon_\Psi} (\delta = 0.060)$} & 0.0739 & 0.0588 & 0.1490 \\
        {$\hat{C^\epsilon_\Psi} (\delta = 0.125)$} & 0.2261 & 0.1702 & 0.2822 \\
    \bottomrule
    \end{tabular}
    \caption{Values of empirical accuracy and stability constant obtained for ReNN, StReNN, and Tikhonov reconstructors, trained on noisy data with $\delta = 0.025$ and tested with different values of $\delta$.}
    \label{tab:table_ReNN}
\end{table}

\subsection{Comparison among reconstructors}\label{sec:experiment_all}

In this final subsection we provide an overview of the results and compare the NN-based reconstructors with the ReNN-based ones.
Figure \ref{fig:stabilityB} shows the output images of the reconstructors trained on noisy data $\y^\delta_i$, $\delta=0.025$ and tested on noisy data with $\delta=0.060$.
As previously observed, the stabilization technique is effective as demonstrated by the image quality and the SSIM value. Interestingly, comparing iNN and ReNN we observe that the ReNN output images inherit smoothness from the regularized images used as target in \eqref{eq:renn_training}, and exhibits a higher SSIM.
At last, ReNN also outperforms Tikhonov reconstructor in terms of SSIM. 

In Figure \ref{subfig:error_ex_B} we plot the reconstruction error of the methods  for increasing value of $\delta \in [0, 0.1]$. The value of $\delta=0.025$ used in the training is indicated with a star marker.
It is discernible that the blue iNN curve demonstrates a markedly steeper gradient, commencing from the minimal error value and escalating to the maximal. The red plot, representing StiNN,  intersects the blue iNN curve at approximately $\delta=0.055$, indicating a more stable behavior at higher noise levels.
The remaining three curves, corresponding to the regularized approaches ReNN, StReNN, and Tikhonov, exhibit similar slopes and behaviors. They manifest elevated errors for smaller values of $\delta$, yet surpass the performance of iNN when $\delta > 0.07$, yielding results comparable to those of StiNN.
Finally,  Figure \ref{subfig:error_ex_B} presents the boxplots of the experimental accuracy achieved across $T=20$ executions with varied random realizations.
The limited variance in these plots indicates that the values of $\hat{\eta}^{-1}$   are remarkably consistent for each individual reconstructor, thereby affirming the robustness of our accuracy definition.

\begin{table}\sffamily
\centering
\begin{tabular}{ccc}
$\y^\delta$ & iNN & StiNN \\
\textit{(SSIM = 0.4484)} & \textit{(SSIM = 0.6063)} & \textit{(SSIM = 0.7203)} \\
\includegraphics[width=0.28\textwidth,trim=50 10 20 60,clip]{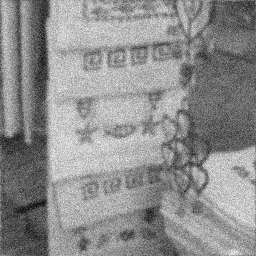 } 
&  \includegraphics[width=0.28\textwidth,trim=50 10 20 60,clip]{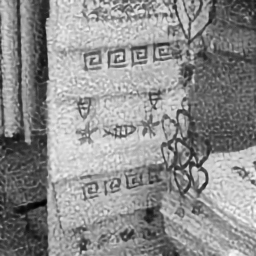} 
& \includegraphics[width=0.28\textwidth,trim=50 10 20 60,clip]{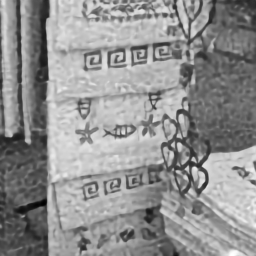} \vspace{2mm} \\
Tik & ReNN & StReNN \\
\textit{(SSIM = 0.6137)} & \textit{ (SSIM = 0.6841)} & \textit{(SSIM = 0.7174)} \\
\includegraphics[width=0.28\textwidth,trim=50 10 20 60,clip]{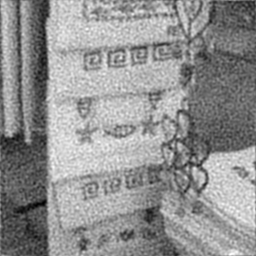} 
& \includegraphics[width=0.28\textwidth,trim=50 10 20 60,clip]{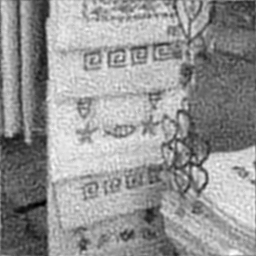} 
& \includegraphics[width=0.28\textwidth,trim=50 10 20 60,clip]{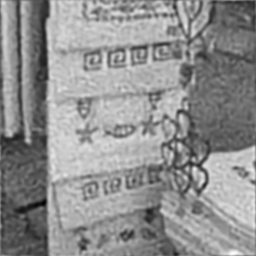} \\
\end{tabular}
\captionof{figure}{Blurred noisy input image $\y^{\delta}$ ($\delta = 0.06$) on the top left and  examples of reconstruction obtained by the iNN, StiNN, ReNN, StReNN and Tikhonov methods on a test image.}
\label{fig:stabilityB}
\end{table} 


\begin{figure}[h]
\centering
\begin{subfigure}{.45\textwidth}
    \includegraphics[width=\textwidth]{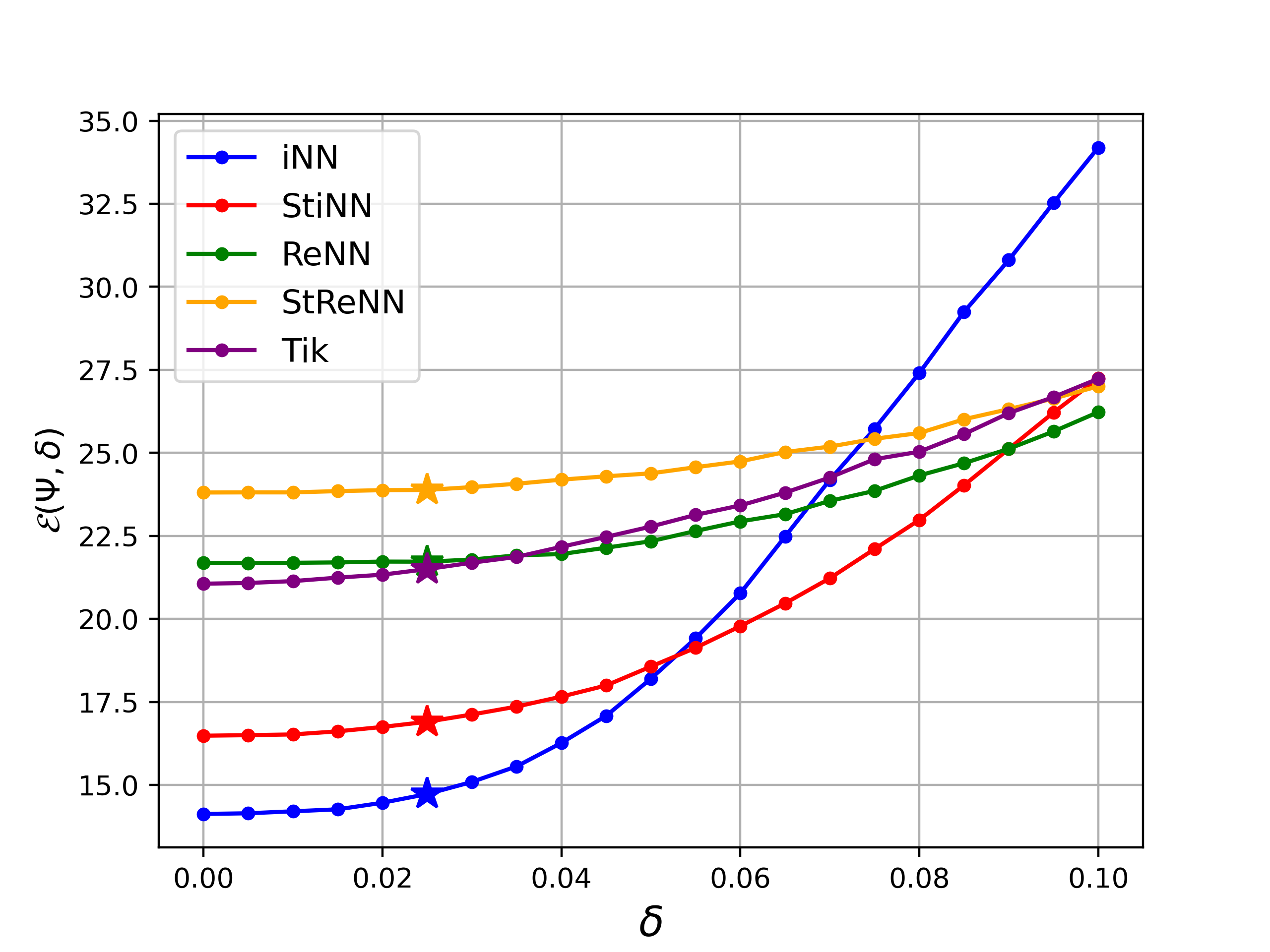}
    \caption{}
    \label{subfig:error_ex_B}
\end{subfigure}
\begin{subfigure}{.45\textwidth}
    \includegraphics[width=\textwidth]{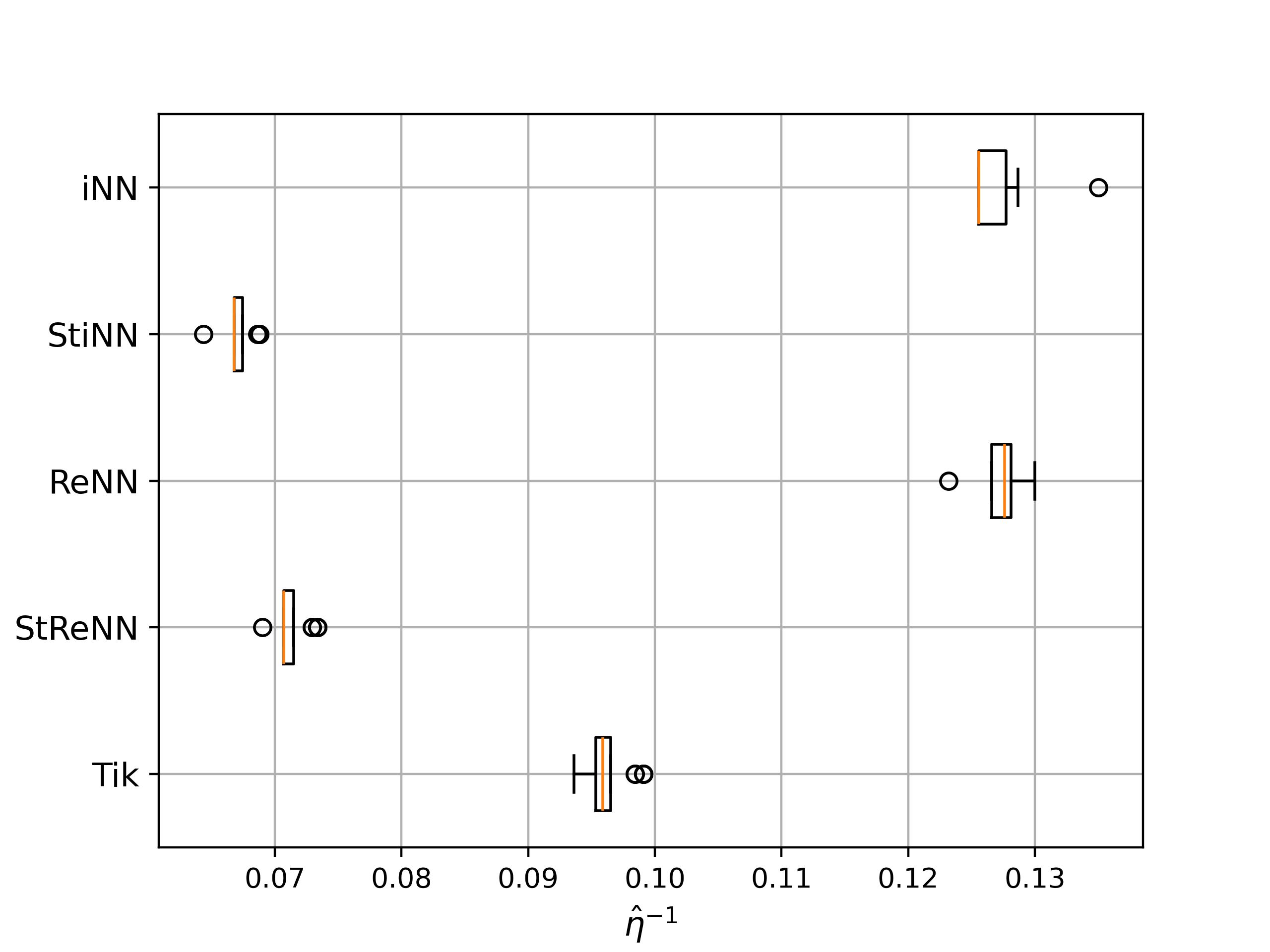}
    \caption{}
    \label{subfig:error_ex_BB}
\end{subfigure}
\caption { (a) Plots of the empirical error yielded by iNN, StiNN, ReNN, StReNN and Tikhonov reconstructors for increasing values of $\delta$ in the test images. (b) Boxplots over the $T=20$ executions.}
\label{fig:results_ex_B}
\end{figure}

\section{Conclusions}\label{sec:concl} 

In this paper, we conducted a comprehensive theoretical analysis of a broad spectrum of reconstructors for addressing a discrete ill-posed inverse problem with noisy data. Our findings, particularly encapsulated in Theorem \ref{thm:stab_vs_acc_tradeoff}, establish that enhancing stability in these reconstructors invariably leads to a decrease in accuracy. Our focus was primarily on reconstructors that leverage neural networks.

In consideration of the trade-off theorem, our objective was to enhance the stability of reconstructors based on deep learning, while preserving their accuracy as much as possible. We based our analysis on the reconstructors represented by the popular end-to-end NN approach for image restoration and we also considered the extensively utilized noise injection stabilization technique, here referred to as iNN. 
As is commonly understood, these  approaches are trained using datasets that include images with known ground truth.

We have proposed new deep learning-based approaches: (i) an additional reconstructor, ReNN, which is trained on noisy images and increases the stability of NN by inheriting regularization from a model-based scheme in its training; (ii) a stabilization technique  which stabilizes the solving process by reducing the impact of the noise with few iterations of a model-based algorithm and it is applied to all the proposed reconstructors resulting in StNN, StiNN and  StReNN.
 
We performed extensive numerical experiments on image deblurring and denoising, with results serving to substantiate the theoretical framework presented in our study.
Firstly, we observe, from Table \ref{tab:table_NN}, Table \ref{tab:table_iNN} and Table \ref{tab:table_ReNN}, that 
the introduction of the proposed stabilizers reduces the stability constants of $99.6\%$ in  StNN, of about  $50\%$ in StiNN and StReNN, with a minimal accuracy loss of about $10-20\%$ in StiNN and StReNN and of about $50\%$ in StNN.
Secondly, in cases where only noisy data are available and ground truth images are not accessible, the ReNN approach performs exceptionally well and represents a more stable alternative compared to the Tikhonov reconstructor, as demonstrated by Figure \ref{tab:table_ReNN} and Figure \ref{fig:results_ex_B}. ReNN outperforms even NN when noise impacts the data.

We believe that this new approach for solving noisy linear inverse problems with stable deep learning-based tools is relevant in this field. It can be further theoretically extended to more general problems and formally applied in real imaging applications, as, for example,  in \cite{evangelista2023ambiguity}.

\subsection*{Funding}
This work has been partially supported by the GNCS - Gruppo Nazionale per il Calcolo Scientifico [{\it ''Apprendimento automatico e tecniche variazionali per la tomografia'' INdAM GNCS Project}, grant code CUP\_E55F55000270001] and by the U.S. National Science Foundation, grant codes DMS-2038118 and DMS-2208294.

\nocite{*}
\bibliographystyle{plain}
\bibliography{biblio}


\end{document}